\documentclass{amsart}
\usepackage{amsmath,amssymb,amsthm}
\usepackage[colorlinks=false, pdfborder={0 0 0}]{hyperref}

\usepackage{amsfonts,amsmath, amsthm, amscd, amssymb, graphicx, color, mathrsfs, stmaryrd}
\usepackage[utf8]{inputenc}
\usepackage{tikz-cd}
\usepackage{adjustbox}
\usepackage{mathtools}
\usepackage{enumerate}
\usepackage{marginnote}

\usepackage[normalem]{ulem}

\usepackage{bbm}

\newcommand{\Complex}{\mathbb C}
\newcommand{\Natural}{\mathbb N}

\newcommand{\N}{\mathbb N}

\newcommand{\set}[1]{\left\{#1\right\}}

\newcommand{\R}{\mathbb{R}}

\renewcommand{\geq}{\geqslant}
\renewcommand{\leq}{\leqslant}

\newcommand{\norm}[1]{\left\Vert#1\right\Vert}

\DeclareMathOperator{\spann}{span}

\newcommand{\NA}{\operatorname{NA}}

\newcommand{\sign}{\operatorname{sign}}

\newcommand{\supp}{\operatorname{supp}}

\newcommand{\ext}{\operatorname{ext}}

\newtheorem{thm}{Theorem}[section]

\newtheorem{prop}[thm]{Proposition}
\newtheorem{coro}[thm]{Corollary}
\newtheorem{lema}[thm]{Lemma}

\newtheorem{question}[thm]{Question}
\theoremstyle{definition}
\newtheorem{defi}[thm]{Definition}
\newtheorem{ejem}[thm]{Example}
\newtheorem{rema}[thm]{Remark}

\numberwithin{equation}{section}

\def\fnote#1{\footnote}

\def\natu{{\mathbb N}}

\def\comp{{\mathbb C}}

\def\K{{\mathbb K}}

\def\ignora#1{}

\def\n3#1{\left\vert  \! \left\vert \! \left\vert \, #1 \, \right\vert \!
  \right\vert \! \right\vert }

\newcommand{\pten}{\ensuremath{\widehat{\otimes}_\pi}}

\newcommand{\ptensor}{\widehat{\otimes}_{\pi}}

\title{Functions in $L_1(\mu,Y)$ with optimal tensor representations}

\author[L. C. García-Lirola]{Luis C. Garc\'ia-Lirola}
\address[L. C. García-Lirola]{Departamento de Matemáticas, Universidad de Zaragoza, 50009, Zaragoza, Spain} 
\email{\texttt{luiscarlos@unizar.es}}
\urladdr{\url{https://personal.unizar.es/luiscarlos/}}

\author[J. Guerrero-Viu]{Juan Guerrero-Viu}
\address[J. Guerrero-Viu]{Departamento de Matemáticas, Universidad de Zaragoza, 50009, Zaragoza, Spain} 
\email{j.guerrero@unizar.es}

\subjclass[2020]{46B04; 46B20; 46B28}

\keywords{Norm-attainment; Projective tensor product; Bochner integrable functions; Strictly convex space.}

\begin{document}

\begin{abstract}
We study the existence and characterization of optimal tensor representations of elements in the space $L_1(\mu,Y)$ of Bochner integrable functions. We completely describe the set of norm-attaining elements in two settings. First, when the Banach space $Y$ is strictly convex, and second, when $Y=L_1(\nu)$ and $\mathbb K=\mathbb R$. In both situations, our analysis yields the existence of non-norm-attaining tensors whenever the underlying measures are not purely atomic. Finally, we introduce a geometric property over $Y$ ensuring that every element in $L_1(\mu, Y)$ admits an optimal representation. In particular, this holds for Lipschitz-free spaces over complete scattered metric spaces, for $C(K)$ spaces when $K$ is a compact Hausdorff totally disconnected space, and for $c_0(\Gamma)$ where $\Gamma$ is any index set. As a byproduct, we settle two open questions regarding projective norm-attainment. 
\end{abstract}

\maketitle

\section{Introduction}

The connection between the existence of norm-attaining functions and the geometric structure of Banach spaces has been extensively studied and has generated numerous contributions (e.g. for bounded linear operators \cite{Bourgain, Lindenstrauss63, Sch2, Uhl}, for bounded bilinear mappings \cite{AAP, AFW, Choi, saleh1}, and for Lipschitz functions \cite{AMCRZT, CCGMR, Godefroy, KMS}). Similarly, the search of optimal representations has played a significant role in the geometry of Banach spaces, giving rise to a substantial body of literature in different contexts (e.g. integrals and convex series of molecules in Lipschitz-free spaces \cite{APS26, APS, AliagaRZ, Weaver}, and the $\lambda$-property or CSRP \cite{AL, ALS, Lohman, Suarez}).

Bridging the gap between those perspectives, projective tensor products offer a natural setting where optimal representations and norm-attainment meet together. Namely, given $u\in X\pten Y$, it is said that $u$ is norm-attaining if the infimum from the definition of its projective norm turns out to be a minimum. In other words, if there exists an optimal representation $u=\sum_{i=1}^\infty x_i\otimes y_i$ such that $\norm{u}_\pi =\sum_{i=1}^\infty \norm{x_i}\norm{y_i}$. In particular,  concerning spaces of integrable (vector-valued) functions and under the isometric identifications $L_1(\mu)\pten Y=L_1(\mu,Y)$, it is natural to ask the following question: given a function $f\in L_1(\mu,Y)$, whether it can be expressed as $f(\omega)=\sum_{i=1}^\infty f_i(\omega)y_i$, where $(f_i)_{i=1}^\infty \subseteq L_1(\mu)$, $(y_i)_{i=1}^\infty \subseteq Y$, and which furthermore satisfies $\norm{f}=\sum_{i=1}^\infty \norm{f_i}\norm{y_i}$. This is clearly equivalent to saying that $f$ has an optimal representation in the projective norm sense. In this sense, the main goal of this paper is to take advantage of the identification $L_1(\mu)\pten Y=L_1(\mu,Y)$, in order to describe the elements from $L_1(\mu,Y)$ that admit optimal representations.

Returning to tensor products, we recall that the notion of projective norm-attainment was introduced in \cite{DJRRZ}, motivated by the development of nuclear operator theory. However, in the finite-dimensional case, the study of norm-attaining tensors goes back to Pelczy\'nsky and Tomczak-Jaegermann, who study it in \cite{PTJ} under the name of faithful operator. After that, new results have appeared in the literature. Writing $\NA_\pi(X\pten Y)$ for the set of norm-attaining tensors, we collect here the main results in this line:
\begin{enumerate}
    \item $\NA_\pi(X\pten Y)=X\pten Y$ if
    \begin{itemize}
        \item     $X$ and $Y$ are finite dimensional \cite[Proposition 3.5]{DJRRZ};
        \item $X=Y$ is a Hilbert space 
    (based on a diagonalization argument, see \cite[Proposition 3.8]{DJRRZ}); 
    \item $X$ is a finite-dimensional polyhedral space and $Y$ is $1$-complemented in its bidual \cite[Theorem 4.1]{DGLJRZ} and \cite[Corollary 3.2]{GGR};
    \item $X=\ell_1(\Gamma)$ for some set $\Gamma$ and $Y$ is any  space \cite[Proposition 3.6]{DJRRZ};
        \item $X$ and $Y$ are dual spaces, one of them with the approximation property, and $X=W^*$ where $W\subseteq Z$ and $Z^*=\ell_1(\Gamma)$ \cite[Theorem 3.4]{GGR}.
    \end{itemize}
    \item $\NA_\pi(X\pten Y)$ is dense in $X\pten Y$ if $X$ has the metric $\pi$-property and $Y$ either has the metric $\pi$-property or it is a dual space with the RNP  \cite[Theorem 4.8]{DJRRZ} and \cite[Corollary 4.3]{GGR}, if $X$ is a polyhedral Banach space with the metric $\pi$-property and $Y$ is a dual Banach space \cite[Theorem 4.2]{DGLJRZ} or if $X$ and $Y$ are dual spaces with the RNP and either $X^*$ or $Y^*$ has the approximation property \cite[Corollary 4.6]{DGLJRZ}.
    \item There are Banach spaces $X$ and $Y$ such that $\NA_\pi(X\pten Y)$ is not dense in $X\pten Y$ \cite[Theorem 5.1]{DJRRZ} (in particular, in such space there are non-norm-attaining tensors which are expressible as finite sum of basic tensors).
    \item $X\pten Y\setminus \NA_\pi(X\pten Y)$ is dense if $X=c_0(\Gamma)$ and $Y$ is a Hilbert space \cite[Theorem 3.3]{rueda23}. The same happens, for instance, if $X=c_0$ and both $Y$ and $Y^*$ are strictly convex (in the real setting) \cite[Remark 4.8]{ADGVJR}. In particular, it is possible that both $\NA_\pi(X\pten Y)$ and its complement are simultaneously dense in $X\pten Y$. 
\end{enumerate}

Nevertheless, we remark that there is no known description of the set $\NA_\pi(X\pten Y)$ in the afore-mentioned results, except that in the case it coincides with the whole space $X\pten Y$. As a result, the task of characterizing elements that admit an optimal representation becomes even more appealing. From this perspective, our study is intended to provide a framework for finding examples and counterexamples of various phenomena, since it offers a setting where one can work conveniently (see, for instance, Remarks \ref{rema:FNA} and \ref{rema:l1strictlyconvex}). In particular, the former resolves an open question raised in \cite{ADGVJR}. 

Let us now describe the organization of this paper. Section \ref{sect:notationpreliminaries} is devoted to introducing the notation we will follow throughout this work and presenting the main definitions and some auxiliary results regarding measure theory and norm-attainment. Furthermore, we include in Remark \ref{Rema:DisjointSuppNA} an elementary proof of the fact that $\NA_\pi(L_1(\mu)\pten Y)$ is always dense regardless of the measure $\mu$ and the Banach space $Y$ (which has been proven with different techniques in \cite[Corollary 4.3]{GGR}). Next, in Section \ref{sect:strictlyconvex}, we focus on the case where $Y$ is strictly convex. We first identify which representations from $L_1(\mu,Y)$ are optimal (Proposition \ref{prop:L1TensorYStrictlyConvex}), yielding two different characterizations of norm-attaining tensors in Theorem \ref{thm:SuppDisjNormAttaining}, one in terms of disjointly supported functions and the other in terms of having an essentially countable range (in the real setting, we obtain in Theorem \ref{thm:NAtotalvariation} another characterization in terms of the total variation of the associated measures). As a consequence, we obtain in Corollary \ref{cor:uniquerepresentation} a unique optimal representation (up to scalar multiplication) of norm-attaining elements. Moreover, Theorem \ref{thm:L1(mu)PurelyAtomic} and Corollary \ref{coro:everytensorinL1muY} show that in this context, every tensor from $L_1(\mu)\pten Y$ is norm-attaining if and only if $\mu$ is purely atomic. We also prove that whenever $\mu$ is not purely atomic and $Y$ is strictly convex, the set $\NA_\pi (L_1(\mu)\pten Y)$ has empty interior and it is not a $G_\delta$ set (Propositions \ref{prop:NAComplementaryDense} and \ref{prop:NAisnotGdelta}). Finally, we show in Proposition \ref{prop:L1SuppDisjointNormAttaining} that in the non-strictly convex case, there are norm-attaining elements that do not verify any of the conditions from Theorem \ref{thm:SuppDisjNormAttaining}. On the other hand, Section \ref{sect:L1L1} is dedicated to the case $L_1(\mu)\pten L_1(\nu)$. Under $\sigma$-finiteness assumptions, Theorem \ref{thm:L1L1characterization} and Corollary \ref{cor:caracterizacionrealvalued} establish that a real-valued element is norm-attaining if and only if certain of its level sets can be expressed as a countable union of measurable rectangles. Furthermore, in its most general form, Theorem \ref{thm:L1L1} shows that every element is norm-attaining if and only if either $\mu$ or $\nu$ is a purely atomic measure. As a final result, Proposition \ref{prop:L1L1meager} ensures that the set of norm-attaining elements is meager, provided that both measures are finite and atomless. Lastly, in Section \ref{sect:CCG}, which concerns only the real setting, we define a geometric condition on $Y$ (see Definition \ref{def:CCG}), and we prove in our main result (Theorems \ref{thm:CSRP} and \ref{thm:quotientwKa}) that it guarantees that every element in $L_1(\mu,Y)$ is norm-attaining, regardless of the measure $\mu$. In particular, we show that this is the case when:
\begin{itemize}
    \item $Y=\mathcal F(M)$, where $M$ is a complete scattered space (Corollary \ref{cor:F(M)}).
    \item $Y=C(K)$, where $K$ is a compact totally disconnected space (Corollary \ref{cor:C(K)}).
    \item $Y=L_\infty(\nu)$, for any measure $\nu$ (Corollary \ref{cor:Linfty}).
    \item $Y=c_0(\Gamma)$, for any set $\Gamma$ (Corollary \ref{cor:c0}).
\end{itemize}
At the end of the paper, we provide a negative answer to an open question from \cite{DJRRZ} which asks if every tensor is norm-attaining whenever one of the factors has property $\alpha$.

\section{Notation and Preliminaries}\label{sect:notationpreliminaries}

Throughout the paper $X$, $Y$, and $Z$ denote a Banach space over the scalar field $\K$, which can be either $\R$ or $\comp$. We reserve the notation $B_X$ and $S_X$ for the  closed unit ball and the unit sphere of a Banach space $X$, respectively. We write $\mathcal L(X,Y)$ for the space of bounded linear operators from $X$ into $Y$. If $Y=\K$, then $\mathcal L(X,\K)$ is denoted by $X^*$ and is known as the topological dual of $X$. 

\subsection{Measure theory}
We understand a \textit{measure space} $(\Omega, \Sigma,\mu)$ as a triple where $\Sigma$ is a $\sigma$-algebra over a set $\Omega$ and $\mu$ is a (non-negative, countably additive) measure. For every measurable subsets $A,B\in \Sigma$ we consider the equivalence relation  $A\sim B$ if $\mu(A\Delta B)=0$, where $A\Delta B\coloneqq (A\cup B)\setminus (A\cap B)$ is the symmetric difference. Throughout the text, measurable sets will be identified with their equivalence classes. This abuse of notation is harmless, since all statements are considered modulo null sets.

A measurable set $A\in \Sigma$ is an \textit{atom} if $0<\mu(A)<\infty$, and for every measurable set $B\subseteq A$, either $\mu(B)=0$ or $\mu(B)=\mu(A)$. Moreover, given $E\in \Sigma$ with $0<\mu(E)<\infty$, we say that $E$ is \textit{purely atomic} if $E\sim \bigcup \{ A\subseteq E: A \text{ is an atom} \}$ (note that since atoms are considered modulo $\sim$, $\mu$ is countably additive and $\mu(E)<\infty$, $E$ can only contain a countable number of atoms so the above set is measurable). Finally, we say that $\mu$ is \textit{purely atomic} if for every $E\in \Sigma$ with $0<\mu(E)<\infty$, we have that $E$ is purely atomic. From this definition, it is clear that if  $\mu$ is not purely atomic, we can find some measurable set $A\in \Sigma$ with $0<\mu(A)<\infty$ such that $A$ does not contain any atom. We remark this fact, because we will use it repeatedly in later sections. Finally, we say that $\mu$ is \textit{atomless} if there are no atoms in $\Sigma$.

Given a measure space $(\Omega, \Sigma,\mu)$ and a Banach space $Y$, we say that $s:\Omega \rightarrow Y$ is a \textit{simple function} if there exists (finitely many) pairwise disjoint measurable sets $E_1,\ldots,E_n\in \Sigma$ with $\mu(E_i)<\infty$ for all $i\in \{1,\ldots,n\}$ such that $s$ is constant on each $E_i$ and vanishes outside $\bigcup_{i=1}^n E_i$. As usual, we denote by $\chi_E$ the \textit{characteristic function} of a measurable set $E\in \Sigma$. Moreover, a function $f:\Omega\rightarrow Y$ is \textit{strongly measurable} or \textit{Bochner measurable} if it satisfies both of the following conditions:
\begin{itemize}
    \item $f$ is \textit{weakly measurable} (i.e. $y^*\circ f$ is $\Sigma$-Borel measurable for every $y^*\in Y^*$),
    \item $f$ is \textit{essentially separably valued} (i.e. there is $N\in \Sigma$ with $\mu(N)=0$ such that $\{f(\omega):\omega\notin N\}$ is contained in a separable subspace of $Y$). 
\end{itemize}
If $f$ also satisfies that $\int_{\Omega} \norm{f(\omega)} \text{ d}\mu<\infty$, then we say that $f$ is \textit{Bochner integrable}. We denote by $L_1(\mu,Y)$ the Banach space of \textit{Bochner integrable functions} from $\Omega$ to $Y$, modulo equality almost everywhere (i.e. $f=g$, $\mu-$a.e. if and only if $\mu(\{\omega\in \Omega: f(\omega)\neq g(\omega)\})=0$). If $Y=\K$, then $L_1(\mu,\K)$ is simply denoted by $L_1(\mu)$. It is well-known that the simple functions form a dense set in $L_1(\mu,Y)$. Given $f\in L_1(\mu,Y)$ we define its \textit{support} as the (measurable) set $$\supp f \coloneqq \{ \omega\in \Omega : f(\omega)\neq 0\}.$$ Observe that $\supp f$ is defined up to a null set, but as commented above (and just as we do for functions in $L_1(\mu,Y)$), we identify $\supp f$ with its equivalence class via $\sim$. Finally, we note that the definition of a purely atomic measure adopted here is justified by the fact that if $\mu$ is purely atomic, then $L_1(\mu)$ is isometrically isomorphic to $\ell_1(\Gamma)$, where $\Gamma$ indexes the distinct atoms (corresponding equivalence classes) of $\mu$ (see e.g. \cite{Phelps}).

Along the paper, we will need to transfer some conditions on $L_1(\mu)$ to the case $L_1([0,1])$. To this end, we will need the following known result (see e.g. \cite[Proposition 9.1.1]{Bogachev}. We denote by $\lambda$ the \textit{Lebesgue measure} on $\R$.

\begin{lema}{\label{lema:Non-ConstantFunctionFiniteNonAtomicCase}}
     Let $(\Omega, \Sigma, \mu)$ be a measure space and assume that $\mu$ is atomless and finite. Then, there exists a measurable function $g \colon \Omega \rightarrow [0,\mu(\Omega)]$
    such that $\mu\circ g^{-1}$ is the Lebesgue measure on $[0,\mu(\Omega)]$. 
\end{lema}

\begin{lema}{\label{lema:Non-ConstantFunctionGeneralCase}}
    Let $(\Omega, \Sigma, \mu)$ be a measure space and $A\in \Sigma$ such that $0<\mu(A)<\infty$. Assume $A$ does not contain any atom. Then, there exists a measurable function 
    $g \colon  \Omega \rightarrow [0,\mu(A)]$
      such that 
         $$\mu\left(A\cap g^{-1}(E)\right) =\lambda(E), \quad \forall E\subseteq [0,\mu(A)] \text{ measurable}.$$
\end{lema}

\begin{proof}
    Consider in $(\Omega, \Sigma)$ the measure $\nu\coloneqq \mu|_A$ (i.e.  $\nu(B)=\mu(A\cap B)$ for all $B\in \Sigma$). By hypothesis over $A$ we have that $\nu$ is a finite atomless measure  on $(\Omega,\Sigma)$. Hence, due to Lemma \ref{lema:Non-ConstantFunctionFiniteNonAtomicCase} there exists a measurable function $g\colon \Omega \rightarrow[0,\nu(A)]$ such that $\mu(A\cap g^{-1}(E))=\nu(g^{-1}(E))=\lambda(E)$. 
\end{proof}

Note that the function $g$ above satisfies that $g|_B$ is not constant for any $B\in \Sigma$ with $B\in A$ and $\mu(B)>0$. Indeed, if $g(B)=\{c\}$, then 
\[ \mu(B)=\mu(A\cap B)\leq \mu(A\cap g^{-1}(\{c\}))=\lambda(\{c\})=0.\]

\subsection{Projective tensor product and norm-attainment}

We recommend \cite{ryan} for a detailed account of projective tensor products. The \textit{projective tensor product} of $X$ and $Y$, denoted by $X \pten Y$, is the completion of the algebraic tensor product $X \otimes Y$ endowed with the norm
$$
\|u\|_{\pi} := \inf \left\{ \sum_{i=1}^n \|x_i\| \|y_i\|: u = \sum_{i=1}^n x_i \otimes y_i \right\},$$
where the infimum is taken over all such representations of $u$.  It is well known that $\|x \otimes y\|_{\pi} = \|x\| \|y\|$ for every $x \in X$, $y \in Y$, and that the closed unit ball of $X \pten Y$ is the closed convex hull of the set $\{ x \otimes y: x \in B_X, y \in B_Y \}$. Observe that every $G\in L(X, Y^*)$ acts on $X \pten Y$ via
$$
G \left( \sum_{i=1}^{n} x_i \otimes y_i \right) = \sum_{i=1}^{n} G(x_i)(y_i),$$
for $\sum_{i=1}^{n} x_i \otimes y_i \in X \otimes Y$. This action establishes a linear isometry from $L(X,Y^*)$ onto $(X\pten Y)^*$ (see e.g. \cite[Theorem 2.9]{ryan}).

Furthermore, special attention will be devoted to the case $X=L_1(\mu)$ for some measure space $(\Omega,\Sigma,\mu)$. In this particular case, it is well-known that the following isometric identification $L_1(\mu)\pten Y=L_1(\mu,Y)$ holds for any Banach space $Y$ \cite[Example 2.19]{ryan}. Indeed, the identification is given by
$$\left(\sum_{i=1}^nf_i\otimes y_i\right) (\omega)=\sum_{i=1}^n f_i(\omega)y_i, \quad \omega\in \Omega,\ \mu-\text{a.e.}$$
This yields an alternative way of computing the projective norm, which will play a key role throughout the paper. We will use this identification repeatedly without explicit mention, and we shall write $\norm{\cdot}$ for either $\norm{\cdot}_\pi$ or $\norm{\cdot}_{L_1(\mu,Y)}$ since these norms coincide and no confusion should arise.

Moreover, the above identification when $\mu$ is taken to be the counting measure, that is $\ell_1(\Gamma)\pten X=\ell_1(\Gamma,X)$, establishes the following consequence: given two Banach spaces $X$ and $Y$, then for every $u\in X\pten Y$ and every $\varepsilon>0$, there exist sequences $(x_i)_{i=1}^\infty$ in $X$ and $(y_i)_{i=1}^\infty$ in $Y$ with $u=\sum_{i=1}^\infty x_i\otimes y_i$ (where the above convergence is in the norm topology of $X\pten Y$) and such that $\Vert u\Vert_{\pi}\leq \sum_{i=1}^\infty \Vert x_i\Vert\Vert y_i\Vert\leq \Vert u\Vert_{\pi}+\varepsilon$. Consequently, it follows that
$$\Vert u\Vert_{\pi}=\inf\left\{\sum_{i=1}^\infty \Vert x_i\Vert\Vert y_i\Vert: \sum_{i=1}^\infty \Vert x_i\Vert\Vert y_i\Vert<\infty,  u=\sum_{i=1}^\infty x_i\otimes y_i \right\}$$
where the infimum is taken over all of the possible representations of $u$ as limit of a series in the above form. Thus, it is natural to introduce the following definition \cite[Definition 2.1]{DJRRZ}.

\begin{defi}
    Let $X,Y$ be Banach spaces and let $u\in X\pten Y$. We say that $u$ \textit{attains its projective norm} if there exist sequences $(x_i)_{i=1}^\infty\subseteq X$, and $(y_i)_{i=1}^\infty\subseteq Y$ such that
    \[u=\sum_{i=1}^\infty x_i\otimes y_i,\quad \text{and} \quad \norm{u}_{\pi}=\sum_{i=1}^\infty \norm{x_i}\norm{y_i}.\]
    In fact, the second  equality happens iff $\norm{z}_\pi=\sum_{i=1}^\infty \norm{x_i\otimes y_i}_\pi$. Furthermore, we say that the representation $z=\sum_{i=1}^\infty x_i\otimes y_i$ is \textit{optimal}.
    We denote $\NA_\pi(X\pten Y)$ the \textit{set of norm-attaining tensors} in $X\pten Y$. 
\end{defi}

In \cite[Corollary 4.3]{GGR} it is obtained, in an indirect way, that $\NA_\pi(L_1(\mu)\pten Y)$ is dense in $L_1(\mu)\pten Y$ for any measure $\mu$ and any Banach space $Y$. Let us provide an elementary proof for that fact based on the afore-mentioned isometric identification $L_1(\mu)\pten Y=L_1(\mu, Y)$. 

\begin{rema}{\label{Rema:DisjointSuppNA}}
    Let $(\Omega, \Sigma, \mu)$ be a measure space and $Y$ be a Banach space. Notice that given  $u=\sum_{i=1}^{\infty} f_i\otimes y_i\in L_1(\mu)\ptensor Y=L_1(\mu, Y)$ such that $\mu(\supp{f_i}\cap \supp{f_j})=0$ for all $i\neq j\in \natu$, we have that $u\in \NA_{\pi}(L_1(\mu)\ptensor Y)$, because 
    \begin{align*}
        \norm{u}_{\pi}&=\norm{u}_{L_1(\mu,Y)} = \int_{\Omega} \norm{\sum_{i=1}^{\infty} f_i(\omega) y_i}_Y \text{ d}\mu =\sum_{i=1}^{\infty} \int_{\supp{f_i}} \norm{f_i(\omega) y_i}_Y\text{ d}\mu \\ &= \sum_{i=1}^{\infty} \norm{ y_i}_Y\int_{\supp{f_i}}  |f_i|\text{ d}\mu = \sum_{i=1}^{\infty} \norm{ y_i}_Y\int_{\Omega}  |f_i|\text{ d}\mu = \sum_{i=1}^{\infty} \norm{ y_i}_Y\norm{f_i}_{L_1(\mu)}.
    \end{align*}
    For the sake of clarity, we have explicitly specified the ambient norm here, although we shall omit this in the sequel whenever no ambiguity arises.
\end{rema}

The remark above shows in particular that $\NA_\pi(L_1(\mu)\pten Y)$ contains the simple functions. Since every function in $L_1(\mu, Y)$ can be approximated by simple ones, this provides a direct proof of the following result, given in \cite[Corollary 4.3]{GGR}.
    
		\begin{prop} Let $(\Omega, \Sigma, \mu)$ be a  measure space and $Y$ be a Banach space. Then,
			\[\overline{\NA_\pi(L_1(\mu)\pten Y)}= L_1(\mu)\pten Y\]
		\end{prop}

\subsection{Auxiliary results}

Finally, we collect several auxiliary lemmata to be used in the sequel. We state them in an abstract setting for the sake of generality, although they will most often be applied in the projective tensor product $L_1(\mu)\pten Y$.

\begin{lema}{\label{Lema:NormAttainingSumFactors2}} Let $X$ be a Banach space and let $x\in X$. Assume that $x=\sum_{i=1}^\infty x_i$ and $\norm{x}=\sum_{i=1}^\infty\norm{x_i}$. Then for every $(\lambda_i)_{i=1}^\infty\subseteq[0,\infty)$ we have $$\norm{\sum_{i=1}^\infty \lambda_i x_i}=\sum_{i=1}^\infty \lambda_i\norm{x_i}.$$
\end{lema}
\begin{proof}
    Let $x^* \in S_{X^*}$ such that $\text{Re }x^*(x)=\norm{x}$. 
    Then $\text{Re }x^*(x_i)=\norm{x_i}$ for every $i$ (if not, we would get $\text{Re } x^*(x)<\norm{x}$).
    Hence, we conclude $$\sum_{i=1}^\infty \lambda_i\norm{x_i}\geq \norm{\sum_{i=1}^\infty \lambda_i x_i}\geq \text{Re }x^*\left(\sum_{i=1}^\infty \lambda_i x_i\right)=\sum_{i=1}^\infty \lambda_i \text{Re }x^*(x_i)=\sum_{i=1}^\infty \lambda_i\norm{x_i}.$$
\end{proof}

We obtain the following immediate consequence, which will be a useful fact in later sections.
 
\begin{coro}\label{coro:NormAttainingSumFactors}
Let $X$ be a Banach space and $x \in X$ with $x=\sum_{i=1}^\infty x_i$. Then $\norm{x}=\sum_{i=1}^\infty \norm{x_i}$ if and only if $\norm{\sum_{i\in F} x_i}=\sum_{i\in F}\norm{x_i}$ for each finite subset $F\subseteq \mathbb N$. 
\end{coro}

We will also need the following lemmas of independent interest.

\begin{lema}\label{l:PairwiseNonColinear}
    Let $X,Y$ be Banach spaces and let $u=\sum_{n=1}^\infty x_n \otimes y_n$ be an optimal representation of $u \in \NA_\pi(X \ptensor Y)$. 
    Let $I \subseteq \Natural$ be maximal such that the vectors $\set{y_{i}:i \in I}$ are pairwise non-colinear. 
    Then there are $(x'_i)_{i\in I} \subseteq X$, such that $u=\sum_{i\in I} x'_i \otimes y_{i}$ is an optimal representation.
\end{lema}
\begin{proof}
 For $i\in I$, let $J_i=\set{j\in \Natural:y_j=c_jy_{i}\mbox{ for some } c_j\in \K}$.
 Notice that the sets $J_i$ are pairwise disjoint and $\Natural=\bigcup_{i\in I} J_i$.
 Since the series $u=\sum_{n=1}^\infty x_n\otimes y_n$ is absolutely convergent we have $$u=\sum_{i \in I} \sum_{j \in J_i} x_j \otimes y_j=\sum_{i \in I} \sum_{j\in J_i} c_jx_j\otimes y_i=\sum_{i\in I} x'_i\otimes y_i$$
 where we set $x'_i:=\sum_{j\in J_i} c_jx_j$.
 Let us check that the series defining $x'_i$ converges absolutely. 
 Indeed using Lemma~\ref{Lema:NormAttainingSumFactors2} we have 
 \[\infty> \norm{\sum_{j\in J_i} x_j\otimes y_j}=\sum_{j\in J_i}\norm{x_j \otimes y_j}=\norm{y_i}\sum_{j\in J_i}\norm{c_jx_j}.\]  
 Finally, we also have $$\norm{x'_i\otimes y_i}=\norm{\sum_{j\in J_i} c_jx_j\otimes y_i}=\norm{\sum_{j\in J_i} x_j\otimes y_j}=\sum_{j\in J_i} \norm{x_j\otimes y_j}$$ so $\norm{u}=\sum_{i\in I}\norm{x'_i\otimes y_i}$.
\end{proof}

\begin{lema}\label{lema:sumadeNAesNA}
    Let $X,Y$ be Banach spaces and let $(u_n)_{n=1}^\infty\subseteq \NA_\pi(X\pten Y)$. If $\norm{\sum_{n=1}^\infty u_n}=\sum_{n=1}^\infty \norm{u_n}<\infty$, then $\sum_{n=1}^\infty u_n\in \NA_\pi(X\pten Y)$.
\end{lema}

\begin{proof}
    For each $n\in \N$, since $u_n$ attains its norm, find $(x_{n,i})_{i=1}^\infty \subseteq X$ and $(y_{n,i})_{i=1}^\infty\subseteq Y$ such that $u_n=\sum_{i=1}^\infty x_{n,i}\otimes y_{n,i}$ and $\norm{u_n}=\sum_{i=1}^\infty \norm{x_{n,i}}\norm{y_{n,i}}$. Hence, we obtain $u=\sum_{n=1}^\infty \sum_{i=1}^\infty x_{n,i}\otimes y_{n,i}$ and
    \begin{align*}
        \norm{\sum_{n=1}^\infty u_n}=\sum_{n=1}^\infty \norm{u_n}= \sum_{n=1}^\infty \sum_{i=1}^\infty \norm{x_{n,i}}\norm{y_{n,i}}.
    \end{align*}
\end{proof}

In the particular case where $X=L_1(\mu)$ we record an elementary consequence of the triangle inequality.         
\begin{lema}\label{Lema:NormAttainingSubsets}
    Let $(\Omega, \Sigma, \mu)$ be a  measure space and $Y$ be a Banach space. Let $u_1,u_2\in L_1(\mu)\pten Y$ and consider $u=u_1+u_2$. Then,
    $$\norm{u}=\norm{u_1}+\norm{u_2} \Longleftrightarrow \norm{u(\omega)}=\norm{u_1(\omega)}+\norm{u_2(\omega)}, \quad \omega \in \Omega, \ \mu-\text{a.e.}$$
\end{lema}

\section{Norm-attaining functions in \texorpdfstring{$L_1(\mu, Y)$}{L1(mu,Y)} with \texorpdfstring{$Y$}{Y} strictly convex}\label{sect:strictlyconvex}

Our first goal in this section is to identify the functions in $\NA_\pi(L_1(\mu)\pten Y)$, for a strictly convex Banach space $Y$. Recall that a Banach space $Y$ is \textit{strictly convex}, if for every $x,y\in S_Y$ with $x\neq y$ we have $\norm{x+y}<2$.

We start by characterizing the optimal representations of elements in $L_1(\mu)\pten Y$. 

\begin{prop}{\label{prop:L1TensorYStrictlyConvex}}
    Let $(\Omega, \Sigma, \mu)$ be a measure space and let $Y$ be a strictly convex Banach space. Let $u=\sum_{i=1}^{\infty} f_i\otimes y_i \in L_1(\mu)\ptensor Y$, such that the vectors $y_i$ are pairwise non-colinear. Then, $$\norm{\sum_{i=1}^{\infty} f_i\otimes y_i}=\sum_{i=1}^{\infty} \norm{f_i\otimes y_i}$$ if and only if $\mu(\supp{f_i}\cap\supp{f_j})=0$, for all $i\neq j\in \natu$.
\end{prop}

\begin{proof}
    The implication from right to left follows from Remark \ref{Rema:DisjointSuppNA}. 

    In order to prove the converse, note that by Corollary \ref{coro:NormAttainingSumFactors} it suffices to consider the case where $u=f_1\otimes y_1+f_2\otimes y_2$ and $\norm{u}=\norm{f_1\otimes y_1}+\norm{f_2\otimes y_2}$. Also, without any loss of generality, we can assume that $\norm{y_1}=\norm{y_2}=1$. 
    
    Now, assume that $\mu(\supp{f_1}\cap \supp{f_2})>0$. Then, define \begin{align*}A&=\{ \omega \in \supp f_1 \cap \supp f_2 :  |f_1(\omega)|\leq |f_2(\omega)|\},\\ B&=\{ \omega \in \supp f_1 \cap \supp f_2 :  |f_1(\omega)|> |f_2(\omega)|\}.\end{align*} 
    It is clear that either $\mu(A)>0$ or $\mu(B)>0$. Switching the roles of $f_1$ and $f_2$ if needed we suppose that $\mu(A)>0$. Define the measurable map $\theta : \Omega \rightarrow \K$ given by $$\theta(\omega)=\begin{cases}
\frac{|f_1(\omega)|f_2(\omega)}{|f_2(\omega)|f_1(\omega)}, & \text{if } \omega \in A,\\
0,  & \text{if } \omega \in \Omega\setminus A,
\end{cases}$$
and observe that $|\theta(\omega)|=1$ and $\frac{f_2(\omega)-\theta(\omega)f_1(\omega)}{\theta(\omega)f_1(\omega)}\in [0,+\infty)$ for $\omega \in A$, $\mu-$a.e.
   On the one hand, we have that 
    \begin{align*}
        &\int_{A} \norm{f_1(\omega)y_1+f_2(\omega)y_2}\text{ d}\mu \\&=\int_{A} \norm{f_1(\omega)(y_1+\theta(\omega)y_2)+(f_2(\omega)-\theta(\omega)f_1(\omega))y_2}\text{ d}\mu \\ &\leq \int_{A}|f_1(\omega)|\norm{y_1+\theta(\omega)y_2}\text{ d}\mu+\norm{y_2}\int_{A}|f_2(\omega)-\theta(\omega)f_1(\omega)|\text{ d}\mu.
    \end{align*}
    On the other hand, due to Lemma \ref{Lema:NormAttainingSubsets} 
    we obtain
    \begin{align*}
        &\int_{A} \norm{f_1(\omega)y_1+f_2(\omega)y_2} \text{ d}\mu \\&=\int_{A} \norm{f_1(\omega)y_1} \text{ d}\mu +\int_{A}\norm{f_2(\omega)y_2} \text{ d}\mu \\&=\norm{y_1}\int_{A} |f_1(\omega)|\text{ d}\mu+  \norm{y_2}\int_{A} \left(|\theta(\omega)f_1(\omega)|+|f_2(\omega)-\theta(\omega)f_1(\omega)|\right)\text{ d}\mu \\&= \left(\norm{y_1}+\norm{y_2}\right)\int_{A} |f_1(\omega)|\text{ d}\mu+  \norm{y_2}\int_{A} |f_2(\omega)-\theta(\omega)f_1(\omega)|\text{ d}\mu,
    \end{align*}
    where the second equality follows from the definition of $A$. Thus, since $Y$ is strictly convex, we have that $\norm{y_1+\theta(\omega)y_2}<\norm{y_1}+\norm{\theta(\omega)y_2}=\norm{y_1}+\norm{y_2}$  for $\omega\in A$, $\mu$-a.e., and the contradiction is obtained.
\end{proof}

Now, we can identify the set of norm-attaining tensors in $L_1(\mu)\pten Y$ for a strictly convex space $Y$. As far as we know, this is the first case when such identification of $\NA_\pi(X\pten Y)$ is achieved  (apart from the ones where every tensor attains its norm). 

\begin{thm}\label{thm:SuppDisjNormAttaining}
    Let $(\Omega, \Sigma, \mu)$ be a measure space and let $Y$ be a strictly convex Banach space. Given $u\in L_1(\mu)\pten Y$, the following assertions are equivalent.
    \begin{itemize}
        \item[i)] $u\in \NA_\pi (L_1(\mu)\pten Y)$.
        \item[ii)] $u=\sum_{i=1}^{\infty} f_i\otimes y_i$ with $\mu\left(\supp{f_i}\cap \supp{f_j}\right)=0$, for all $i\neq j$. 
        \item[iii)] There is a sequence $(y_n)_{n\in \N}\subseteq S_Y$ such that $$\mu\left\{ \omega \in \Omega : u(\omega)\notin \bigcup_{n\in \N} \spann\{y_n\}\right\}=0.$$
    \end{itemize}
\end{thm}

\begin{proof}
    It is clear that ii) $\Rightarrow$ iii). \\Now, to show that     
    i) $\Rightarrow$ ii), pick $u\in \NA_{\pi}(L_1(\mu)\ptensor Y)$. Hence, there is some representation $\sum_{i=1}^\infty f_i \otimes y_i$ such that $\norm{u}=\sum_{i=1}^\infty \norm{f_i}\norm{y_i}$. We may assume that the vectors $y_i$ are pairwise non-colinear, by Lemma \ref{l:PairwiseNonColinear}. 
    Thus, it follows from Proposition \ref{prop:L1TensorYStrictlyConvex}.

    For iii) $\Rightarrow$ i), suppose that there is a sequence $(y_n)_n\subseteq S_Y$ satisfying that $\mu\left\{ \omega \in \Omega : u(\omega)\notin \bigcup_{n\in \N} \spann\{y_n\}\right\}=0.$ Let $I\subseteq \N$ be maximal such that the vectors $(y_i)_{i\in I}$ are pairwise non-colinear. Without loss of generality we may assume that $\norm{y_i}=1$, for all $i\in I$. Consider $(y_i^*)_{i\in I}\subseteq S_{Y^*}$ a sequence of functionals satisfying $y_i^*(y_i)=1$, for all $i\in I$. Moreover, for each $i\in I$, define the continuous map $f_i \colon Y\rightarrow [0,+\infty)$ given by $$f_i(y)=\norm{y}-|y_i^*(y)|, \quad \forall y\in Y.$$ Hence, consider the measurable sets $A_i\coloneqq (f_i\circ u)^{-1}(\{0\})$, for each $i\in I$ (observe that $u$ is strongly measurable and $f_i$ is continuous so the composition yields a $\Sigma$-Borel map).  Since $Y$ is strictly convex, it follows that $\mu (A_i\Delta B_i)=0$, where $$B_i=\{\omega\in \Omega : u(\omega)\in \spann\{y_i\}\}, \quad \forall i\in I.$$ It is clear $\mu(B_i\cap B_j)=\mu\{ \omega\in \Omega: u(\omega)=0\}$ for $i\neq j$. Writing $g_i=(y_i^*\circ u)\chi_{B_i}$ for all $i\in I$, and observing that $\mu(\supp g_i\cap \supp g_j)=0$ for all $i\neq j$, we conclude  $$\norm{\sum_{i\in I} g_i\otimes y_i}=\sum_{i\in I} \norm{g_i}\norm{y_i},$$
    thanks to Remark \ref{Rema:DisjointSuppNA}. Finally, the equality $u=\sum_{i\in I} g_i\otimes y_i$ is an easy verification. 
\end{proof}

We will see later in Proposition \ref{prop:L1SuppDisjointNormAttaining} that the statement of Theorem \ref{thm:SuppDisjNormAttaining} does not hold in the non-strictly convex case.

Next corollary asserts that for each norm-attaining tensor, there is a unique optimal representation up to scalar multiplication.

\begin{coro}\label{cor:uniquerepresentation}
    Let $(\Omega, \Sigma, \mu)$ be a measure space and let $Y$ be a strictly convex Banach space. Let $u\in \NA_\pi (L_1(\mu)\pten Y)$. Then, there is a set $I\subseteq \N$, a sequence of pairwise non-colinear vectors  $(y_i)_{i\in I}\subseteq S_Y$ and a sequence of non-zero functions $(f_i)_{i\in I}\subseteq L_1(\mu)$, such that $$u=\sum_{i\in I} f_i\otimes y_i \quad \text{ and } \quad \norm{u}=\sum_{i\in I} \norm{f_i}.$$
    Indeed, for each $i\in I$, the function $f_i\in L_1(\mu)$ is given by $f_i=(y_i^*\circ u)\chi_{E_i}$, where $y_i^*\in S_{Y^*}$ is any norming functional of $y_i$ (i.e. $y_i^*(y_i)=1$) and $E_i\subseteq \Omega$ is the measurable set $\{\omega \in \Omega : u(\omega)\in \spann\{y_i\}\}$.
    
    Furthermore, the above representation is unique up to scalar multiplication.
\end{coro}

\begin{proof}
    The existence of such a representation follows from the proof of Theorem~\ref{thm:SuppDisjNormAttaining}. To prove the uniqueness, suppose there are two optimal representations $$u=\sum_{i\in I} f_i\otimes y_i \quad \text{ and } \quad u=\sum_{j\in J} g_j\otimes z_j.$$
     Applying Proposition \ref{prop:L1TensorYStrictlyConvex}, we deduce that $\mu(\supp f_i \cap \supp f_j)=0$ for $i, j\in I, \ i\neq j$ and $\mu(\supp g_i\cap \supp g_j)=0$ for $i,j\in J, \ i\neq j$. 
    Since $\mu(\supp u \Delta \bigcup_{j\in J} \supp g_j)=0$, and the functions $f_i$ are all non-zero, it is clear that for each $i\in I$, there is some $j\in J$ satisfying that $\mu(\supp f_i \cap \supp g_j)>0$. We claim that this relation defines a bijection $\varphi : I\rightarrow J$. 

    First of all, it is well defined because if there is some $i\in I$ and $j_1,j_2\in J$ such that $\mu(\supp f_i\cap \supp g_{j_1})>0$ and $\mu(\supp f_i\cap \supp g_{j_2})>0$, then 
        \begin{align*}
            f_i(\omega)y_i&=u(\omega)=g_{j_1}(\omega)z_{j_1}, \quad \omega\in \supp f_i\cap \supp g_{j_1}, \ \mu-\text{a.e.},
            \\  f_i(\omega)y_i&=u(\omega)=g_{j_2}(\omega)z_{j_2}, \quad \omega\in \supp f_i\cap \supp g_{j_2}, \ \mu-\text{a.e.},
        \end{align*}
        so $z_{j_1},z_{j_2}\in \spann\{y_i\}$ which is impossible since the vectors $z_j$ are pairwise non-colinear. Thus, $\varphi : I \rightarrow J$ is well-defined. Analogously to the above, switching the roles of $I$ and $J$, it is easy to prove that $\varphi$ is injective. Finally, it must be surjective because $\mu(\supp u \Delta \bigcup_{i\in I} \supp f_i)=0$ and the functions $g_j$ are all non-zero.
         Once we have proven that $\varphi$ is a bijection, it follows that $\mu(\supp f_i \Delta \supp g_{\varphi(i)})=0$ for all $i\in I$. Furthermore, for each $i\in I$, we have
        \begin{equation}\label{Eq1}
            f_i(\omega)y_i=u(\omega)=g_{\varphi(i)}(\omega)z_{\varphi(i)}, \quad \omega \in \supp f_i, \ \mu-\text{a.e.},
        \end{equation}
        so $z_{\varphi(i)}\in \spann\{y_i\}$. Hence, there is some $\theta_i\in \K$ with $\theta_iy_i=z_{\varphi(i)}$. Thus, $|\theta_i|=1$ and $\frac{1}{\theta_i}f_i=g_{\varphi(i)}$.
\end{proof}

As a consequence, we can provide a counterexample to the following question from \cite[Section 2.3 (p. 11)]{ADGVJR}:
$$\textit{Is $\text{FNA}_\pi (X\pten Y)=\NA_\pi(X\pten Y)\cap (X\otimes Y)$, for all $X,Y$?}$$
where $\text{FNA}_\pi (X\pten Y)$ denotes the set of \textit{finitely norm-attaining tensors} (i.e. elements that attain their norm at some finite representation).
\begin{rema}\label{rema:FNA}
    We are going to show that there exist finite-rank norm-attaining tensors that only attain their norm at a series of infinite terms. Indeed, let $Y$ be a strictly convex Banach space with $\dim(Y)\geq 2$. Consider a sequence $(f_n)_{n\in \N}\subseteq L_1([0,1])$ of non-zero functions satisfying that $\mu(\supp f_i \cap \supp f_j)=0$, for all $i\neq j$, and $\sum_{n=1}^\infty \norm{f_n}<\infty$, and two non-colinear vectors $y_1,y_2\in S_Y$. Then, $$u=\sum_{n=1}^\infty f_n \otimes \left(y_1+\frac{1}{n}y_2\right) =  \left(\sum_{n=1}^\infty f_n\right)\otimes y_1 +  \left(\sum_{n=1}^\infty \frac{1}{n}f_n\right)\otimes y_2\in L_1([0,1])\otimes Y.$$ In fact, $\norm{u}=\sum_{n=1}^\infty \norm{f_n}\norm{y_1+\frac{1}{n}y_2}$ since $\mu(\supp f_i \cap \supp f_j)=0$ for all $i\neq j$. Thus, $u\in \NA_\pi (L_1([0,1])\pten Y)$. However, it follows from Corollary \ref{cor:uniquerepresentation}, that every optimal representation of $u$ must contain an infinite number of elementary tensors.
\end{rema}

In the case $\K=\R$, we can relate the property of norm-attaining with a different concept. Recall that given a finitely additive vector measure $\nu \colon \Sigma\to Y$, its \textit{total variation} or \textit{total 1-variation} (see e.g. \cite[Definition 4.1]{CRRSP} and \cite[p. 93]{ryan}) is defined as $$|\nu|_1(\Omega)\coloneqq \sup \left\{ \sum_{i=1}^n \norm{\nu(A_i)}: \{A_1,\ldots,A_n\} \text{ is a partition of } \Omega \right\}.$$
We will say that $\nu$ has \textit{bounded total variation} if $|\nu|_1(\Omega)$ is finite.

Furthermore, for each $u\in L_1(\mu,Y)$ we can define the vector measure $\nu_u: \Sigma\to Y$ given by $$\nu_u(A)\coloneqq \int_A u(\omega) \text{ d}\mu, \quad \forall A\in \Sigma.$$ It can be seen that $\nu_u$ is finitely additive, has bounded total variation, and $|\nu_u|_1(\Omega)=\norm{u}$ (see e.g. \cite[Theorem 4.3]{CRRSP}), so we can consider the following notion.

\begin{defi}
    Let $(\Omega,\Sigma,\mu)$ be a measure space and $Y$ be a Banach space. Given $u\in L_1(\mu,Y)$ we say that it \textit{attains its total variation} if there is a sequence of pairwise disjoint measurable sets $(A_n)_{n=1}^\infty\subseteq \Sigma$ such that $$|\nu_u|_1(\Omega)=\sum_{n=1}^\infty \norm{\nu_u(A_n)}.$$
\end{defi}
Thus, we can state the following result, which in particular shows the equivalence between attaining the projective norm and the total variation, provided the target space is real and strictly convex. Recall that $L_\infty(\mu,Y^*)\subseteq L_1(\mu,Y)^*$ isometrically under the natural duality mapping. Moreover, a measurable map $f$ is said to be \textit{essentially countably valued} if there is some $N\in \Sigma$ with $\mu(N)=0$ such that the set $\{f(\omega) : \omega \notin N\}$ is countable.

\begin{thm}\label{thm:NAtotalvariation}
    Let $(\Omega, \Sigma, \mu)$ be a measure space and let $Y$ be a real strictly convex Banach space. Given $u\in L_1(\mu)\pten Y$, the following assertions are equivalent.
    \begin{itemize}
        \item[i)] $u\in \NA_\pi (L_1(\mu)\pten Y)$.
        \item[ii)] $u=\sum_{i=1}^{\infty} f_i\otimes y_i$ with $\mu\left(\supp{f_i}\cap \supp{f_j}\right)=0$, for all $i\neq j$. 
        \item[iii)] $P\circ u$ is essentially countably valued, where $P:Y\rightarrow S_Y$ is given by $P(y)=\frac{y}{\norm{y}}$ if $y\neq 0$ and $P(0)=0$.
        \item[iv)] There is an essentially countably valued $g\in S_{L_\infty(\mu,Y^*)}$ with $\langle g,u\rangle=\norm{u}$.
        \item[v)] $u$ attains its total variation.
    \end{itemize}
\end{thm}

\begin{proof}
    i) $\Leftrightarrow$ ii) $\Leftrightarrow$ iii) were proven in Theorem \ref{thm:SuppDisjNormAttaining} (notice that, in the real case, condition iii) of the present theorem is equivalent to condition iii) of Theorem \ref{thm:SuppDisjNormAttaining}).

    iii) $\Rightarrow$ iv) Let $P\circ u=\sum_{n=1}^\infty \chi_{A_n}\otimes \widetilde{y}_n$, where $(\widetilde{y}_n)_{n=1}^\infty \subseteq S_Y$ and $(A_n)_{n=1}^\infty \subseteq \Sigma$ are pairwise disjoint. Consider for each $n\in \N$, a functional $y_n^*\in S_{Y^*}$ such that $y_n^*(\widetilde{y}_n)=1$ and define $g\coloneqq \sum_{n=1}^\infty \chi_{A_n}\otimes y_n^*.$ It is clear that $g\in S_{L_\infty(\mu,Y^*)}$ and $\langle g, u\rangle =\norm{u}$.

    iv) $\Rightarrow$ v) Let $g=\sum_{n=1}^\infty \chi_{A_n} \otimes y_n^*$, where $(y_n^*)_{n=1}^\infty \subseteq B_{Y^*}$ and $(A_n)_{n=1}^\infty \subseteq \Sigma$ are pairwise disjoint. Therefore,
    \begin{align*}
        \norm{u}&=\langle g, u\rangle = \sum_{n=1}^\infty \left(\int_{A_n} y_n^*(u(\omega)) \text{ d}\mu\right) = \sum_{n=1}^\infty y_n^*\left(\int_{A_n} u(\omega) \text{ d}\mu\right)\\&\leq \sum_{n=1}^\infty \norm{\int_{A_n} u(\omega) \text{ d}\mu} \leq |\nu_u|_1(\Omega)=\norm{u},
    \end{align*}
    which yields $$\sum_{n=1}^\infty \norm{\nu_u(A_n)}=\sum_{n=1}^\infty \norm{\int_{A_n} u(\omega)\text{ d}\mu}=|\nu_u|_1(\Omega).$$

    v) $\Rightarrow$ iii) Let $(A_n)_{n=1}^\infty \subseteq \Sigma$ the sequence of pairwise disjoint measurable sets such that $|\nu_u|_1(\Omega)=\sum_{n=1}^\infty \norm{\nu_u(A_n)}$. Replacing each $A_n$ by $A_n\cap \supp u$ we may assume that $\mu(\supp u \Delta(\bigcup_{n=1}^\infty A_n))=0$. Hence, pick $(y_n^*)_{n=1}^\infty \subseteq S_{Y^*}$ such that $$y_n^*\left(\int_{A_n} u(\omega) \text{ d}\mu \right)=\norm{\int_{A_n} u(\omega) \text{ d}\mu}=\norm{\nu_u(A_n)}, \quad \forall n\in \N.$$ Then,
    \begin{align*}
        \norm{u}&=|\nu_u|_1(\Omega) = \sum_{n=1}^\infty \norm{\nu_u(A_n)}= \sum_{n=1}^\infty y_n^*\left(\int_{A_n} u(\omega) \text{ d}\mu \right) \\&= \sum_{n=1}^\infty \left(\int_{A_n} y_n^*(u(\omega)) \text{ d}\mu \right)\leq \sum_{n=1}^\infty \int_{A_n}\norm{u(\omega)} \text{ d}\mu =\norm{u}.
    \end{align*}
    This clearly implies that for each $n\in \N$, we have
    $$y_n^*(u(\omega))=\norm{u(\omega)}, \quad \omega\in A_n,\ \mu-\text{a.e.}$$
    Since $Y$ is strictly convex it yields that there is some $y_n\in S_Y$ such that $$\frac{u(\omega)}{\norm{u(\omega)}}= y_n, \quad \omega\in A_n, \ \mu-\text{a.e.}$$
    In other words, the map $P\circ u$ is essentially countably valued.
\end{proof}

\begin{rema}
    Let us see that for any given $u\in L_1(\mu,Y)$, attaining the projective norm and attaining the total variation may not be equivalent if a) $\K=\Complex$ or b) $Y$ is not strictly convex.
    \begin{itemize}
        \item[a)] If $\K=\Complex$, consider the map $u\in L_1([0,1],\Complex)$ given by $u(t)=e^{i2\pi t}$, for $t\in [0,1]$. It is clear that $u\in \NA_\pi(L_1([0,1])\pten \Complex)$ since $\Complex$ is 1-dimensional. However, let us show that $u$ does not attain its total variation. Suppose there is a sequence $(A_n)_{n=1}^\infty$ of pairwise disjoint Lebesgue measurable sets in $[0,1]$, such that
        $$1=\norm{u}=|\nu_u|_1([0,1])=\sum_{n=1}^\infty |\nu_u(A_n)|.$$ From the triangle inequality it follows that for each $n\in \N$, we have
        $$\lambda(A_n) = |\nu_u(A_n)|=\left|\int_{A_n} e^{i2\pi t} \text{ d}\lambda \right|.$$ Hence, writing $z_n=\frac{1}{\lambda(A_n)}\overline{\int_{A_n} e^{i2\pi t} \text{ d}\lambda}$, we obtain
        \begin{align*}
            \lambda(A_n)&=z_n \int_{A_n} e^{i2\pi t} \text{ d}\lambda=\text{Re }\left(z_n \int_{A_n} e^{i2\pi t} \text{ d}\lambda\right) \\&= \text{Re }\left(\int_{A_n} z_ne^{i2\pi t} \text{ d}\lambda\right) = \int_{A_n} \text{Re }\left(z_n e^{i2\pi t}\right) \text{ d}\lambda,
        \end{align*}
        so $1=\text{Re }\left(z_n e^{i2\pi t}\right)$, for $t\in A_n$, $\lambda-$a.e., because $|z_n e^{i2\pi t}|\leq 1$. This implies that $e^{i2\pi t}=\overline{z_n}$, for $t\in A_n$, $\lambda-$a.e. which yields  $\lambda(A_n)=0$, obtaining the contradiction. 
        \item[b)] We will see in the last section (Corollary \ref{cor:F(M)}) that $\NA_\pi(L_1(\mu)\pten \mathcal F(M))=L_1(\mu)\pten \mathcal F(M)$ for any measure $\mu$, whenever $M$ is a complete scattered metric space (clearly $\mathcal{F}(M)$ is not strictly convex). However, it is known that if every element in $L_\infty(\mu,Y^*)$ is a James boundary for $L_1(\mu,Y)$ and $\mu$ is a probability measure non purely atomic, then $Y^*$ has the Radon-Nikod\'{y}m Property \cite[Teorema A]{CP}. Furthermore, it is easy to see that iv) $\Leftrightarrow$ v) from Theorem \ref{thm:NAtotalvariation} holds for any Banach space $Y$, so if every element from $L_1(\mu,Y)$ attains its total variation then $L_\infty(\mu,Y^*)$ is, in particular, a James boundary. Thus, since $\mathcal{F}(M)^*$ fails the RNP (see e.g. \cite[Theorem~5]{CDW}), we conclude that for a probability non purely atomic measure $\mu$, there are norm-attaining tensors in $L_1(\mu,\mathcal F(M))$ that do not attain their total variation.
    \end{itemize}
\end{rema}

Our next goal is to determine whether every element in $L_1(\mu)\pten Y$ is a norm-attaining tensor. 

\begin{thm}{\label{thm:L1(mu)PurelyAtomic}}
    Let $(\Omega, \Sigma, \mu)$ be a measure space and $Y$ be a Banach space with $\dim(Y)\geq 2$. Then, $$L_1(\mu)\ptensor Y=  \left\{ u=\sum_{i=1}^{\infty} f_i\otimes y_i \in L_1(\mu)\ptensor Y: \mu\left(\supp{f_i}\cap \supp{f_j}\right)=0, \forall i\neq j \right\}$$ if and only if $\mu$ is purely atomic.
\end{thm}

\begin{proof} 
    If $\mu$ is purely atomic and $(E_\gamma)_{\gamma\in \Gamma}$ is the family of (disjoint) atoms of $\mu$, then $L_1(\mu)\pten Y =\ell_1(\Gamma)\pten Y=\ell_1(\Gamma, Y)$. Thus, for any $u\in \ell_1(\Gamma,Y)$ there is some countable subset $I\subseteq \Gamma$ such that $u=\sum_{i\in I} \chi_{E_i} \otimes y_i$, for some $(y_i)_{i\in I}\subseteq Y$. 

    Conversely, suppose that $\mu$ is not purely atomic. That is, there exists $A\in \Sigma$ with $0<\mu(A)<\infty$ such that $A$ contains no atoms. 
    Without loss of generality, we can assume that $\mu(A)=1$ and consider $g\colon \Omega \rightarrow [0,1]$ the nowhere constant function from Lemma \ref{lema:Non-ConstantFunctionGeneralCase}. Replacing $g$ by $g\chi_A$, we may assume that $g|_{\Omega\setminus A}=0$. 
    Next, pick $e_1,e_2\in S_Y$ being non-colinear and consider $u\colon \Omega \rightarrow Y$ given by
    $$u(\omega)= g(\omega)e_1+(\chi_A(\omega)-g(\omega))e_2, \quad  \omega \in \Omega, \ \mu-\text{a.e.} $$
    It is clear that $\mu(\supp{u}\Delta A)=0$  and $\norm{u(\omega)}\leq 1$, for $\omega\in A$,  $\mu-$a.e. Therefore, $u\in L_1(\mu,Y)=L_1(\mu)\pten Y$ and by hypothesis we can write $u$ as
    $$u=\sum_{i=1}^{\infty} f_i\otimes y_i, \quad \text{with} \quad  \mu(\supp{f_i}\cap \supp{f_j})=0,\ \forall i\neq j,$$
     where $f_i\in L_1(\mu)$ and $y_i\in Y$, for every $i\in \N$. Fix $i$ such that $\mu(\supp{f_{i}})>0$ and $y_i\neq 0$. 
    Hence, $u(\omega)= f_{i}(\omega)y_{i}$, for $\omega \in \supp f_i$, $\mu-$a.e., so in particular $\mu((\supp f_i\cap A)\Delta \supp f_i)=0$. Take the unique $\alpha_1, \alpha_2\in \K$ such that $\alpha_1 e_1 + \alpha_2 e_2 =y_i$, and then, for $\omega\in \supp f_i\cap A$, $\mu-$a.e., we have $$g(\omega)=f_i(\omega)\alpha_1, \quad \text{ and } \quad 1-g(\omega)=f_i(\omega)\alpha_2.$$ Thus, $\mu\left\{ \omega \in \supp f_i\cap A \ : \ g(\omega) = \frac{\alpha_1}{\alpha_1+\alpha_2} \right\}=\mu(\supp f_i\cap A)>0$. Therefore, $g$ is constant in some measurable subset of $A$ with positive measure, which gives the contradiction we were looking for.
\end{proof}

As a consequence, we get a characterization of the existence of non-norm-attaining tensors in $L_1(\mu, Y)$ for a strictly convex space $Y$. This extends  \cite[Example 3.12]{DJRRZ}, where it is considered the case when $Y=\ell_2^2$ and when $Y$ is a dual strictly convex space without the Radon-Nikodým property (RNP).
The non-strictly convex case is considered in \cite[Theorem 4.9]{ADGVJR}, showing that there are non-norm-attaining tensors provided  $Y$ is a dual space with the RNP and its predual contains a strictly convex subspace of dimension 2.

\begin{coro}\label{coro:everytensorinL1muY}
    Let $(\Omega, \Sigma, \mu)$ be a  measure space and let $Y$ be a strictly convex Banach space with $\dim(Y)\geq 2$. 
    The following statements are equivalent:
    \begin{itemize}
    \item[i)] $\NA_{\pi}(L_1(\mu)\ptensor Y)=L_1(\mu)\ptensor Y$. 
    \item[ii)] $\mu$ is purely atomic.
    \end{itemize}
\end{coro}

\begin{proof}
    It follows from Theorems \ref{thm:SuppDisjNormAttaining} and \ref{thm:L1(mu)PurelyAtomic}.
\end{proof}

Recall that in \cite[Corollary 3.11]{DJRRZ} it is shown that if $\NA_\pi(X\pten Y)=X\pten Y$ then $\overline{\NA(X, Y^*)}=\mathcal L(X, Y^*)$. Thanks to Corollary \ref{coro:everytensorinL1muY}, we get examples showing that the converse statement does not hold: take $X=L_1([0,1])$ and $Y$ a Banach space such that $Y^*$ is a strictly convex space with the RNP. Then $\overline{\NA(X, Y^*)}=\mathcal L(X, Y^*)$ by \cite{Uhl} but $\NA_\pi(X\pten Y)\neq X\pten Y$.

In the following results, we will take advantage of the explicit description of $\NA_\pi(L_1(\mu)\pten Y)$ to analyze its topological properties in the case that there are non-norm-attaining tensors. First, we show that its complement is indeed dense in $L_1(\mu)\pten Y$. This compares to \cite[Theorem 3.3]{rueda23} and \cite[Remark 4.8]{ADGVJR}, where it is shown, for instance, that both $\NA_\pi(c_0\pten L_p)$ and its complement are dense in $c_0\pten L_p$ (for $1<p<\infty$, in the real setting, and for $p=\frac{2n}{2n-1}, \ n\in \N$, in the complex setting).

\begin{prop}\label{prop:NAComplementaryDense}
    Let $(\Omega, \Sigma, \mu)$ be a measure space and let $Y$ be a strictly convex Banach space with $\dim(Y)\geq 2$. If $\mu$ is not purely atomic, then $$\overline{L_1(\mu)\pten Y \setminus \NA_\pi(L_1(\mu)\pten Y)}=L_1(\mu)\pten Y.$$
\end{prop}

\begin{proof}
    Let $u\in L_1(\mu)\pten Y=L_1(\mu,Y)$ and $\varepsilon>0$. Using that simple functions are dense in $L_1(\mu,Y)$, we can find some $s=\sum_{n=1}^N \chi_{E_n} \otimes y_n\in L_1(\mu,Y)$ such that $\norm{u-s}<\frac{\varepsilon}{2}$.
    Since $\mu$ is not purely atomic, there is some  measurable set $A$ with $0<\mu(A)<\infty$ and satisfying that $A$ contains no atoms. Take such an $A$ with $\mu(A)<\frac{\varepsilon}{2(1+\max\{\norm{y_n}: 1\leq n\leq N\})}$ and consider $g:\Omega \rightarrow [0,\mu(A)]$ the function associated with $A$, given by Lemma \ref{lema:Non-ConstantFunctionGeneralCase}. Replacing $g$ by $g\chi_A$, we may assume that $g|_{\Omega\setminus A}=0$. Then, pick $e_1,e_2\in S_Y$ non-colinear vectors and define the tensor $u'=s|_{\Omega\setminus A} + g\otimes e_1+(\chi_A-g)\otimes e_2\in L_1(\mu)\pten Y$. First, observe that
    \begin{align*}
        \norm{u'-s}&=\int_{A} \norm{s(\omega)-(g(\omega)e_1+(1-g(\omega))e_2)} \text{ d}\mu \\&\leq (1+\max\{\norm{y_n} : 1\leq n\leq N\})\mu(A)<\frac{\varepsilon}{2}.
    \end{align*}
    Furthermore, let us show that $u'$ does not attain its norm. Otherwise, there should be a representation $u'=\sum_{i=1}^\infty f_i\otimes z_i$ with $\mu(\supp f_i \cap \supp f_j)=0$ for all $i\neq j$, by Theorem \ref{thm:SuppDisjNormAttaining}. Hence, we could find some $i\in \N$ such that $\mu(\supp f_i \cap A)>0$. Therefore,
    \begin{equation*}
        g(\omega)e_1+(1-g(\omega))e_2=u'(\omega)=f_i(\omega)z_i, \quad \text{ for } \omega\in \supp f_i \cap A, \ \mu-\text{a.e.},
    \end{equation*}
    from which we would obtain a contradiction as we did in the proof of Theorem \ref{thm:L1(mu)PurelyAtomic}. Thus, $u'\notin \NA_\pi (L_1(\mu)\pten Y)$ and $\norm{u-u'}<\varepsilon$.
\end{proof}

The following remark highlights the existence of tensors that do not admit optimal representations in their projective tensor product, but do admit them in a larger projective tensor product, into which the former embeds as a closed subspace.
\begin{rema}\label{rema:l1strictlyconvex}
    Recall that $\ell_1$ contains a strictly convex two-dimensional subspace $Y$ (see  below \cite[Theorem 5]{Lindenstrauss64}). 
Therefore, by Proposition \ref{prop:NAComplementaryDense}, if $\mu$ is not purely atomic,  the set of non-norm-attaining tensors $L_1(\mu)\pten Y \setminus \NA_\pi(L_1(\mu)\pten Y)$ is dense in $L_1(\mu)\pten Y$. 
However, we have $\NA_\pi (L_1(\mu) \ptensor \ell_1) = L_1(\mu) \ptensor \ell_1$ 
(\cite[Proposition 3.6]{DJRRZ}), so for every $u\notin L_1(\mu)\pten Y \setminus \NA_\pi (L_1(\mu)\pten Y)$, we know that $u$ admits an optimal representation in the bigger space $L_1(\mu) \ptensor \ell_1$. 
Moreover, $L_1(\mu) \ptensor Y$ is (isometrically) a subspace of $L_1(\mu) \ptensor \ell_1$, so such a representation must necessarily include at least one elementary tensor 
$x \otimes y$ with $y \in \ell_1 \setminus Y$.

Furthermore, the example above shows that the assumption in \cite[Lemma 3.1]{GGR} 
cannot be weakened to only require that $X \ptensor Z$ is (isometrically) a subspace of $X \ptensor Y$.
\end{rema}

We can say a bit more about the descriptive complexity of the set $\NA_\pi(L_1(\mu)\pten Y)$. The next result shows that it is not a $G_\delta$ set. This is related to the question of whether $\NA_\pi(X\pten Y)$ may be residual, asked by A. Rueda Zoca at the conference \emph{Lluís Santaló School 2023}. 

\begin{prop}\label{prop:NAisnotGdelta} Let $(\Omega,\Sigma, \mu)$ be 
a measure space and let $Y$ be a strictly convex Banach space with $\dim(Y)\geq 2$. If $\mu$ is not purely atomic, then $\NA_\pi(L_1(\mu)\pten Y)$ is not a $G_\delta$ set. 
\end{prop}

\begin{proof}
    Consider the Cantor set $2^\mathbb N$ endowed with the product topology and the subset $F\subseteq 2^{\mathbb N}$ of finitely supported sequences. Note that $F$ is countable and so it is a meager subset of $2^{\mathbb N}$. We claim that there is a continuous 
    $\Phi\colon 2^\mathbb N\to L_1(\mu)\pten Y$ such that $\Phi^{-1}(\NA_\pi(L_1(\mu)\pten Y))=F$. Since $F$ is meager, this will provide the desired conclusion. 

    Now we prove the claim. To this end, take $A\in \Sigma$ such that $0<\mu(A)<\infty$ and satisfying that $A$ contains no atoms. We may assume that $\mu(A)=1$. Let $g\colon \Omega\to [0,1]$ be the measure-preserving function provided by Lemma \ref{lema:Non-ConstantFunctionGeneralCase}. Take also two non-colinear vectors $y_1,y_2\in S_Y$ and consider the continuous map $\gamma\colon[0,1]\to S_Y$ given by
    \[ \gamma(t)=\frac{(1-t)y_1+ty_2}{\norm{(1-t)y_1+ty_2}}\]
    For $a\in 2^{\mathbb N}$, let $s_a\in L_1([0,1])$ given by 
    \[s_a(t)=\sum_{n=1}^\infty \frac{a_n}{3^n} \varepsilon_n(t)\]    
    where $\varepsilon_n(t)$ denotes the $n$-th binary digit of $t$ (note that this is well-defined up to a null set). Finally, let
    \[u_a(\omega)=\begin{cases} \gamma(s_a(g(\omega))) & \text{if } \omega \in A\\ 0 &\text{otherwise}\end{cases}\]
    Then $u_a\in L_1(\mu)\pten Y$ and we may define $\Phi(a)=u_a$. We will show that $\Phi$
satisfies the required properties.    

First, to check that $\Phi$ is continuous, let $\varepsilon>0$, and take $\delta>0$ such that $\norm{\gamma(t)-\gamma(s)}<\varepsilon$ whenever $|t-s|<\delta$. Let $N$ so that $\sum_{n> N} 3^{-n}<\delta$. If $a_n=b_n$ for $n\leq N$, then
\[\norm{s_a-s_b}_\infty\leq \sum_{n>N}\frac{|a_n-b_n|}{3^n} <\delta \]
and so
\begin{align*}
\norm{u_a-u_b}= \int_A \norm{\gamma(s_a(g(\omega)))-\gamma(s_b(g(\omega)))}d\omega \leq \varepsilon
\end{align*}

It remains to show that $\Phi^{-1}(\NA_\pi(L_1(\mu)\pten Y))=F$. Given $a\in 2^\mathbb N$,  Theorem \ref{thm:SuppDisjNormAttaining} and the fact that   $\norm{u_a(\omega)}=1$ for each $\omega$ show that $u_a\in \NA_\pi(L_1(\mu)\pten Y)$ if and only if $\{\omega: u_a(\omega)\}$ is contained in countable union of one-dimensional subspaces, up to a null set. Since $\operatorname{span}\{\gamma(t)\}\neq \operatorname{span}\{\gamma(s)\}$ for $t\neq s$, it follows that  $u_a\in \NA_\pi(L_1(\mu)\pten Y)$ if and only if there exists $C\subseteq[0,1]$ countable such that
\[ 0 = \mu\{\omega: s_a(g(\omega))\notin C\}= \lambda\{t: s_a(t)\notin C\}\]
Clearly, this holds if and only if $a$ is finitely supported. That ends the proof.
\end{proof}

We finish the section showing that the statement of Theorem \ref{thm:SuppDisjNormAttaining} does not hold in the non-strictly convex case.

\begin{prop}{\label{prop:L1SuppDisjointNormAttaining}}
    Let $(\Omega, \Sigma, \mu)$ be a measure space with $\mu$ not purely atomic and let $Y$ be Banach space with $\dim(Y)\geq 2$. Then,
    $$\NA_{\pi}(L_1(\mu)\ptensor Y) = \left\{ u=\sum_{i=1}^{\infty} f_i\otimes y_i : \mu\left(\supp{f_i}\cap \supp{f_j}\right)=0, \forall i\neq j \right\}$$
    if and only if $Y$ is strictly convex. Otherwise, the inclusion from right to left is strict.
\end{prop}

\begin{proof}
    If $Y$ is strictly convex, the result follows from Theorem \ref{thm:SuppDisjNormAttaining}. Conversely, suppose that $Y$ is not strictly convex. Then, there is some non-colinear vectors $e_1,e_2\in S_Y$ such that $\lambda e_1+(1-\lambda)e_2\in S_Y$, for all $\lambda \in [0,1]$. Since $\mu$ is not purely atomic, we can find some $A\in \Sigma$ with $0<\mu(A)<\infty$ such that $A$ does not contain any atom. Without loss of generality, we can assume that $\mu(A)=1$ and consider $g\colon \Omega \rightarrow [0,1]$ the associated function from Lemma \ref{lema:Non-ConstantFunctionGeneralCase}. Moreover, replacing $g$ by $g\chi_A$, we may assume that $g|_{\Omega\setminus A}=0$. Hence, define $f\colon\Omega \rightarrow Y$ given by
    $$f(\omega)= g(\omega)e_1+(\chi_A(\omega)-g(\omega))e_2, \quad \forall \omega \in \Omega.$$
    Observe that
    \begin{align*}
        \norm{f}=\int_{\Omega} \norm{f(\omega)} \text{ d}\mu = \int_{A} \norm{g(\omega)e_1+(1-g(\omega))e_2} \text{ d}\mu = \int_{A} 1 \text{ d}\mu =1.
    \end{align*}
     Furthermore, \begin{align*}
        \norm{g\otimes e_1}+ \norm{(\chi_A-g)\otimes e_2}&=\int_{\Omega} \norm{g(\omega)e_1} \text{ d}\mu + \int_{\Omega} \norm{(\chi_A(\omega)-g(\omega))e_2} \text{ d}\mu\\ &= \norm{e_1}\int_{A} g(\omega) \text{ d}\mu + \norm{e_2}\int_{A} (1-g(\omega)) \text{ d}\mu\\& = \int_{A} (g(\omega)+1-g(\omega) )\text{ d}\mu=1
    \end{align*}
    Thus, $f\in \NA_{\pi} (L_1(\mu)\ptensor Y)$. Finally, if $f= \sum_{i=1}^{\infty} f_i \otimes y_i$ with $\mu(\supp{f_i}\cap \supp{f_j})=0$, for $i\neq j$, we would get the same contradiction as the one obtained in the proof of Theorem \ref{thm:L1(mu)PurelyAtomic}.
\end{proof}

\section{Norm-attainment in \texorpdfstring{$L_1(\mu)\pten L_1(\nu)$}{L1(mu)xL1(nu)}}\label{sect:L1L1}

In this section, we will study the set of norm-attaining tensors $\NA_\pi(L_1(\mu)\pten L_1(\nu))$ for measures $\mu$ and $\nu$. In most of the results of the section, we consider $\sigma$-finite measures, except in Theorem \ref{thm:L1L1}, which is stated for arbitrary measures. The reason for this restriction is that the product measure is not uniquely determined when the underlying measures are not $\sigma$-finite. In that case,  moreover, fundamental results such as Fubini’s theorem may fail to hold (see Notes and comments from \cite[252]{Fremlin}). 

Observe that, we have an isometric identification $L_1(\mu)\pten L_1(\nu)=L_1(\mu\otimes \nu)$, where $\mu\otimes \nu$ represents the product measure (e.g. \cite[253F and 253L]{Fremlin}). This identification yields a new way of computing the projective norm, which will be instrumental in establishing the main results of this section.

Recall first that for a non purely atomic measure $\mu$ it is known that 
\[ \NA_\pi(L_1(\mu)\pten L_1(\mu)) \neq L_1(\mu)\pten L_1(\mu)\]
This follows from \cite[Proposition 3.10]{DJRRZ} and the fact that the set of norm-attaining bilinear forms on $L_1(\mu)\times L_1(\mu)$ is not dense (this was proven in \cite{Choi} for the Lebesgue measure and extended in \cite{saleh1, saleh2} to non purely atomic measures). Here we will describe the elements in $\NA_\pi(L_1(\mu)\pten L_1(\nu))$ and show explicitly that there exist non-norm-attaining tensors when $\mu$ and $\nu$ are not purely atomic.  

We start with several noteworthy lemmata that will be used later to establish our main results. The first one provides a characterization of the optimal representations in the space $L_1(\mu)\pten L_1(\nu)$.
We denote by $\arg(\cdot)$ the\textit{ principal argument} from $\K\setminus\{0\}$ into $(-\pi, \pi]$.

\begin{lema}\label{lema:NAL1L1}
Let $(\Omega_1, \Sigma_1,\mu)$ and $(\Omega_2, \Sigma_2, \nu)$ be $\sigma$-finite measure spaces and $f\in L_1(\mu)\pten L_1(\nu)$. A representation $f=\sum_{i=1}^\infty f_i\otimes g_i $ is optimal if and only if, for all $i,j\in \N$, we have
$$ \arg(f_i(\omega)g_i(\gamma))=\arg(f_j(\omega)g_j(\gamma)),$$ for $(\omega,\gamma)\in (\supp f_i\times \supp g_i)\cap(\supp f_j\times \supp g_j),\ (\mu\otimes\nu)-$a.e.
\end{lema}

\begin{proof}
    The implication from right to left follows from the monotone convergence theorem and the fact that if $a_1,\ldots, a_n\in \K\setminus\{0\}$ satisfy $\arg(a_i)=\arg(a_j)$, for all $i, j\in \{1,\ldots,n\}$, then $\sum_{i=1}^n|a_i| = \left|\sum_{i=1}^n a_i \right|$.
    
    Conversely, suppose that $f\in \NA_{\pi}(L_1(\mu)\pten L_1(\nu))$, so there is some representation $f=\sum_{i=1}^\infty f_i\otimes g_i$ such that $\norm{f} = \sum_{i=1}^\infty \norm{f_i}\norm{g_i}$. Using Corollary \ref{coro:NormAttainingSumFactors}, we obtain that $\norm{f_i\otimes g_i + f_j\otimes g_j} = \norm{f_i}\norm{g_i}+\norm{f_j}\norm{g_j}$ for all $i,j\in \N$. By Fubini's theorem \cite[252C and 252R]{Fremlin} it follows that 
    $$\left| f_i(\omega)g_i(\gamma)+f_j(\omega)g_j(\gamma) \right| = |f_i(\omega)g_i(\gamma)| +|f_j(\omega)g_j(\gamma)|, \ (\omega,\gamma)\in \Omega_1\times \Omega_2,\ (\mu\otimes \nu)-\text{a.e.}$$
    for all $i,j\in \N$.
    Finally,  we conclude that for all $i,j\in \N$ we have $\arg(f_i(\omega)g_i(\gamma))=\arg(f_j(\omega)g_j(\gamma))$ for $(\omega,\gamma)\in (\supp f_i\times \supp g_i)\cap(\supp f_j\times \supp g_j)$, $(\mu\otimes\nu)-$a.e.
\end{proof}

Next, we present a necessary condition for an element to be norm-attaining, in terms of a decomposition of its support into a countable union of measurable rectangles.

\begin{lema}\label{lema:normattainingL1L1}
    Let $(\Omega_1, \Sigma_1,\mu)$ and $(\Omega_2, \Sigma_2, \nu)$ be $\sigma$-finite measure spaces. If $f\in \NA_\pi (L_1(\mu)\ptensor L_1(\nu))$, then there is a sequence of measurable rectangles $(A_i\times B_i)_{i=1}^\infty\subseteq \Sigma_1\times \Sigma_2$ such that
    $$(\mu\otimes \nu)\left(\supp{f}\Delta\left(\bigcup_{i=1}^\infty A_i\times B_i\right)\right)=0.$$
    Furthermore, if $f$ is real-valued, we may choose the sets $A_i$ and $B_i$ so that
      \\$\sign{f}|_{A_i\times B_i}$ is constant for each $i\in\natu$.
\end{lema}

\begin{proof}
    For the first part of the proof, pick $f \in \NA_\pi (L_1(\mu)\ptensor L_1(\nu))$. Thanks to Lemma \ref{lema:NAL1L1}, there is a representation $f=\sum_{i=1}^\infty f_i\otimes g_i$ such that $$\arg(f_i(\omega)g_i(\gamma))=\arg(f_j(\omega)g_j(\gamma))$$ for $(\omega,\gamma)\in (\supp f_i\times \supp g_i)\cap(\supp f_j\times \supp g_j), \ (\mu\otimes\nu)-$a.e., and for all  $i,j\in \N$. Hence, it is easy to see that  $$(\mu\otimes\nu)\left(\supp{f}\Delta\left(\bigcup_{i=1}^\infty \supp{f_i}\times \supp{g_i}\right)\right)=0$$ and $\arg{f}|_{\supp{f_i}\times \supp{g_i}}= \arg{f_i\otimes g_i}$, for all $i\in \N$. 
    
    For the second part of the proof (assuming that $f$ is real-valued), we just need to adjust the measurable rectangles to satisfy the additional condition. Certainly, we can define, for each $i\in\natu$, the following measurable sets
    \begin{align*}
        C_i &=\{ \omega\in \Omega_1 : f_i(\omega)>0\}
        \\ D_i &=\{ \omega\in \Omega_1 : f_i(\omega)<0\}
        \\ E_i &=\{ \gamma\in \Omega_2 : g_i(\gamma)>0\}
        \\ F_i &=\{ \gamma\in \Omega_2 : g_i(\gamma)<0\}
    \end{align*}
    and it is then clear that $\bigcup_{i=1}^\infty \left((C_i\times E_i)\cup(C_i\times F_i)\cup(D_i\times E_i)\cup (D_i\times F_i)\right)$ coincides with $\supp f$, up to a null set. Furthermore, in each of those rectangles, $\sign f$ is obviously constant. 
\end{proof}

We are going to provide a characterization of norm-attainment for non-negative functions. However, the following lemma shows that we may easily extend this characterization to any real-valued function.

\begin{lema}\label{lema:urealsiiu+u-}
     Let $(\Omega_1, \Sigma_1,\mu)$, and $(\Omega_2, \Sigma_2, \nu)$ be $\sigma$-finite measure spaces, and let $f\in L_1(\mu\otimes \nu)$ a real-valued function. Then, $$f\in \NA_\pi(L_1(\mu)\pten L_1(\nu)) \Longleftrightarrow \{f^+, f^-\}\subseteq \NA_\pi(L_1(\mu)\pten L_1(\nu)),$$
     where $f^+$ and $f^-$ denote the usual positive and negative parts of $f$. 
\end{lema}

\begin{proof}
    If $f^+$ and $f^-$ attain their norm, then $f=f^+-f^-$ is also norm-attaining by Lemma \ref{lema:sumadeNAesNA}. Conversely, suppose that $f\in \NA_\pi(L_1(\mu)\pten L_1(\nu))$. We prove that $f^+$ attains its norm, and an analogous argument shows that $f^-$ also attains its norm. Lemma \ref{lema:normattainingL1L1} implies that there are measurable rectangles $(A_i\times B_i)_{i=1}^\infty \subseteq\Sigma_1\times \Sigma_2$ such that $(\mu\otimes\nu)\left(\supp f^+\Delta\left( \bigcup_{i=1}^\infty A_i\times B_i\right)\right)=0$. Up to taking a pairwise disjoint version of those rectangles (observe that the complement of a measurable rectangle and the finite intersection of measurable rectangles can always be expressed as finite unions of measurable rectangles), we may assume that they are so. Furthermore, let $f=\sum_{j=1}^\infty f_j\otimes g_j$ be an optimal representation of $f$. Then, it is clear that
    \begin{align*}
        f^+=f\chi_{\supp f^+} =\sum_{j=1}^\infty \sum_{i=1}^\infty f_j\chi_{A_i}\otimes g_j\chi_{B_i}.
    \end{align*}
    Moreover,
    \begin{align*}
        \sum_{j=1}^\infty \sum_{i=1}^\infty\norm{f_j\chi_{A_i}}\norm{g_j\chi_{B_i}} &= \sum_{j=1}^\infty \sum_{i=1}^\infty \left(\int_{A_i}|f_j(\omega)| \text{ d}\mu\right)\left(\int_{B_i}|g_j(\gamma)| \text{ d}\nu\right) \\&= \sum_{j=1}^\infty \sum_{i=1}^\infty \int_{A_i\times B_i}  |f_j(\omega)||g_j(\gamma)| \text{ d}(\mu\otimes \nu) \\&= \sum_{j=1}^\infty \int_{\supp f^+} |f_j(\omega)||g_j(\gamma)| \text{ d}(\mu\otimes \nu) \\&= \int_{\supp f^+} \sum_{j=1}^\infty |f_j(\omega)||g_j(\gamma)| \text{ d}(\mu\otimes \nu) \\&= \int_{\supp f^+} |f(\omega,\gamma)| \text{ d}(\mu\otimes \nu) = \norm{f^+},
    \end{align*}
    so $f^{+}\in \NA_\pi(L_1(\mu)\pten L_1(\nu))$.
\end{proof}

We now present one of the main results of this section, which provides two different characterizations of norm-attainment for non-negative functions (see also Corollary \ref{cor:caracterizacionrealvalued} for real-valued functions).

\begin{thm}\label{thm:L1L1characterization}
    Let $(\Omega_1, \Sigma_1,\mu)$, and $(\Omega_2, \Sigma_2, \nu)$ be $\sigma$-finite measure spaces, and let $f\in L_1(\mu\otimes \nu)$ with $f\geq 0$. Then, the following assertions are equivalent:
    \begin{itemize}
        \item[i)] $f\in \NA_\pi(L_1(\mu)\pten L_1(\nu))$.
        \item[ii)] $f=\sum_{i=1}^\infty c_i (\chi_{A_i}\otimes \chi_{B_i})$, for some $(c_i)_{i=1}^\infty\subseteq [0,\infty)$ and some measurable rectangles $(A_i\times B_i)_{i=1}^\infty \subseteq \Sigma_1\times \Sigma_2$.
        \item[iii)] For each $c\geq 0$, the level set $\{(\omega,\gamma)\in \Omega_1\times \Omega_2 : f(\omega,\gamma)>c\}$ coincides with a countable union of measurable rectangles, up to a $(\mu\otimes \nu)$-null set.
    \end{itemize}
\end{thm}

\begin{proof}
    i) $\Rightarrow$ ii) If $f$ attains its norm, there are sequences $(f_i)_{i=1}^\infty \subseteq L_1(\mu)$ and $(g_i)_{i=1}^\infty \subseteq L_1(\nu)$, such that $f=\sum_{i=1}^\infty f_i\otimes g_i$ and $\norm{f}=\sum_{i=1}^\infty \norm{f_i}\norm{g_i}$. Since $f\geq 0$, Lemma \ref{lema:NAL1L1} yields $f_i(\omega)g_i(\gamma)=|f_i(\omega)| |g_i(\gamma)|$, $(\mu\otimes\nu)-$a.e., for all $i\in \N$. Hence, it follows that $f=\sum_{i=1}^\infty |f_i|\otimes |g_i|$, so we may assume that $f_i\geq0$ and $g_i\geq 0$, for all $i\in \N$. Therefore, for each $i\in \N$, we can express the previous functions as
    $$f_i=\sum_{j=1}^\infty \alpha_{i,j} \chi_{A_{i,j}} \quad \text{and} \quad g_i=\sum_{k=1}^\infty \beta_{i,k}\chi_{B_{i,k}},$$
    for some $(\alpha_{i,j})_{j=1}^\infty, (\beta_{i,k})_{k=1}^\infty \subseteq[0,\infty)$, $(A_{i,j})_{j=1}^\infty \subseteq \Sigma_1$, and $(B_{i,k})_{k=1}^\infty \subseteq\Sigma_2$
    (observe that we just have to consider, for each one, an increasing sequence of non-negative simple functions $(s_n)_n$ and express the limit function as $\sum_{n\in \N} (s_{n+1}-s_n)$). Thus, we obtain 
    \begin{align*}
        f&=\sum_{i=1}^\infty f_i\otimes g_i = \sum_{i=1}^\infty \left(\sum_{j=1}^\infty\alpha_{i,j} \chi_{A_{i,j}}  \right)\otimes\left(\sum_{k=1}^\infty \beta_{i,k}\chi_{B_{i,k}}\right) \\&=\sum_{i=1}^\infty \sum_{j=1}^\infty \sum_{k=1}^\infty \alpha_{i,j}\beta_{i,k}(\chi_{A_{i,j}}\otimes \chi_{B_{i,k}}),
    \end{align*}
    and by rearranging the indices, the conclusion follows.

    ii) $\Rightarrow$ iii) Fix $c\geq 0$, and observe that since the sequence $(c_i)_{i=1}^\infty$ is non-negative we have that
    $$\{(\omega,\gamma)\in \Omega_1\times \Omega_2 : f(\omega,\gamma)>c\} = \bigcup_{n=1}^\infty \{(\omega,\gamma)\in \Omega_1\times \Omega_2 : s_n(\omega,\gamma)>c\},$$
    where $s_n\coloneqq \sum_{i=1}^n c_i (\chi_{A_i}\otimes \chi_{B_i})$, for all $n\in \N$. Hence, it suffices to show that given $n\in \N$, the set $\{(\omega,\gamma)\in \Omega_1\times \Omega_2 : s_n(\omega,\gamma)>c\}$ is a countable union of measurable rectangles up to a null set. However, this is clear because $s_n$ is measurable with respect to the $\sigma$-algebra generated by the family $\{ A_i\times B_i : 1\leq i \leq n\}$, and every element from this $\sigma$-algebra is a finite union of measurable rectangles.  

    iii) $\Rightarrow$ i) We first prove it  for characteristic functions, secondly for non-negative simple functions, and finally for any $f\geq 0$.
    
    If $f=\chi_E$ and it satisfies iii) then $E$ is a countable union of measurable rectangles, up to a $(\mu\otimes \nu)$-null set. Let $(A_i\times B_i)_{i=1}^\infty\subseteq\Sigma_1\times\Sigma_2$ be a pairwise disjoint version of those rectangles. Therefore, they satisfy that $\mu(E\Delta \bigcup_{i=1}^\infty (A_i\times B_i))=0$. Finally, $f=\chi_E= \sum_{i=1}^\infty \chi_{A_i}\otimes \chi_{B_i}$ and this representation is optimal thanks to Lemma \ref{lema:NAL1L1}.
    
    Suppose now that $f$ is a non-negative simple function and it satisfies iii). Let $\lambda_1,\ldots, \lambda_n$ be the finite number of values (different from 0)  that $f$ takes, and assume without loss of generality that $0<\lambda_1<\lambda_2 <\cdots <\lambda_n.$ Define the sets $E_1=\supp f$ and
    $E_i\coloneqq \{(\omega, \gamma) : f(\omega,\gamma)>\lambda_{i-1}\}$ for all  $i\in \{2,\ldots,n\}.$ From above we know that every $\chi_{E_i}$ attains its norm and so by Lemma \ref{lema:sumadeNAesNA}, we have that $$f=\lambda_1\chi_{E_1}+(\lambda_2-\lambda_1)\chi_{E_2}+\cdots+ (\lambda_n-\lambda_{n-1})\chi_{E_n}$$ also attains its norm.
    
    Finally, let $f\in L_1(\mu\otimes\nu)$ with $f\geq 0$ and satisfying iii). Consider an enumeration $(q_i)_{i=1}^\infty$ of $\mathbb{Q}_{\geq 0}$, and find for each $i\in \N$, a sequence of measurable rectangles $(R_{i,j})_{j=1}^\infty\subseteq \Sigma_1\times \Sigma_2$ such that $\{(\omega,\gamma)\in \Omega_1\times\Omega_2 : f(\omega,\gamma)>q_i\}$ coincides with $\bigcup_{j=1}^\infty R_{i,j}$, up to a $(\mu\otimes \nu)$-null set. Define for each $n\in \N$, the non-negative simple function $s_n\in L_1(\mu\otimes\nu)$, given by
    $$s_n(\omega,\gamma)\coloneqq \max_{1\leq i,j\leq n} q_i \chi_{R_{i,j}}(\omega,\gamma), \quad (\omega,\gamma)\in \Omega_1\times \Omega_2, \ (\mu\otimes \nu)-\text{a.e.}$$
    It is easy to see that $(s_n)_{n=1}^\infty$ is non-decreasing and converges pointwise to $f$. Thus, considering $r_n\coloneqq s_{n+1}-s_n$ and by the monotone convergence theorem, it follows that $f=\sum_{n=1}^\infty r_n$. Furthermore, given $n\in \N$, observe that $r_n$ is measurable with respect to the $\sigma$-algebra generated by $\{ R_{i,j} : 1\leq i,j \leq n+1\}$, and every element from this $\sigma$-algebra is a finite union of measurable rectangles. Thus, the set $\{ (\omega,\gamma) \in \Omega_1\times \Omega_2 : r_n(\omega,\gamma)>c\}$ is a finite union of measurable rectangles, for each $c\geq 0$. Hence, $r_n$ attains its norm and by Lemma \ref{lema:sumadeNAesNA}, we have that $f$ does so.
\end{proof}

Observe that with the aid of Lemma \ref{lema:urealsiiu+u-}, we can completely characterize norm-attaining real-valued functions as a straightforward consequence of Theorem \ref{thm:L1L1characterization}.

\begin{coro}\label{cor:caracterizacionrealvalued}
      Let $(\Omega_1, \Sigma_1,\mu)$, and $(\Omega_2, \Sigma_2, \nu)$ be $\sigma$-finite measure spaces, and let $f\in L_1(\mu\otimes \nu)$ be a real-valued function. Then, the following assertions are equivalent. 
    \begin{enumerate}
        \item[i)] $f\in \NA_\pi(L_1(\mu)\pten L_1(\nu))$.
        \item[ii)] $f=\sum_{i=1}^\infty c_i (\chi_{A_i}\otimes \chi_{B_i})-\sum_{j=1}^\infty d_j(\chi_{C_j}\otimes \chi_{D_j})$, for some $(c_i)_{i=1}^\infty, (d_j)_{j=1}^\infty\subseteq [0,\infty)$ and some measurable rectangles $(A_i\times B_i)_{i=1}^\infty, (C_j\times D_j)_{j=1}^\infty \subseteq \Sigma_1\times \Sigma_2$, satisfying the additional condition $(\mu\otimes \nu)((A_i\times B_i) \cap (C_j\times D_j))=0$, whenever $c_i d_j\neq 0$.
        \item[iii)] For each $c\geq 0$, both of the sets $\{(\omega,\gamma)\in \Omega_1\times \Omega_2 : f(\omega,\gamma)>c\}$ and $\{(\omega,\gamma)\in \Omega_1\times \Omega_2 : f(\omega,\gamma)<-c\}$ coincide with a countable union of measurable rectangles, up to a $(\mu\otimes \nu)$-null set.
    \end{enumerate}
\end{coro}

\begin{rema}
    Observe that given $f,g\in L_1(\mu\otimes\nu)$ real-valued functions, if $f$ and $g$ are norm-attaining, we also have that $\max\{f,g\}$ and $\min\{f,g\}$ attain their norm. This follows from Corollary \ref{cor:caracterizacionrealvalued} and the fact that the family of countable unions of measurable rectangles is stable under finite unions and intersections (we can express the level sets of $\max\{f,g\}$ and $\min\{f,g\}$ in terms of the ones of $f$ and $g$). Thus, in the real setting, we conclude that $\NA_\pi(L_1(\mu)\pten L_1(\nu))$ is a sublattice of $L_1(\mu\otimes\nu)$.
\end{rema}

For later use in Theorem \ref{thm:L1L1}, we state the following immediate corollary from Theorem \ref{thm:L1L1characterization} regarding the particular case of characteristic functions. It should be compared with \cite[Corollary~2.4]{saleh2}.

\begin{coro}\label{coro:equivalencecharacteristicfunctions} Let $(\Omega_1, \Sigma_1,\mu)$, and $(\Omega_2, \Sigma_2, \nu)$ be $\sigma$-finite measure spaces and $E\subseteq \Omega_1\times \Omega_2$ be a $(\mu\otimes\nu)$-measurable set with $(\mu\otimes \nu)(E)<\infty$.  The following statements are equivalent:
\begin{itemize}
    \item[i)] $\chi_E\in \NA_\pi(L_1(\mu)\pten L_1(\nu))$,
    \item[ii)] there is a sequence of measurable rectangles $(A_i\times B_i)_{i=1}^\infty\subseteq \Sigma_1\times \Sigma_2$ such that $(\mu\otimes\nu)(E\Delta(\bigcup_{i=1}^\infty A_i\times B_i))=0$. 
\end{itemize}
\end{coro}

Finally, we will use the above characterization to prove that there are non-norm-attaining elements in $L_1(\mu)\pten L_1(\nu)$ whenever $\mu$ and $\nu$ are both non purely atomic. Before proceeding, we need an auxiliary lemma that should be compared with \cite[Lemma 2]{Choi} and \cite[Lemma 2.1]{saleh1}.

\begin{lema}\label{lema:setwithoutrectangles}
    Let $(\Omega_1, \Sigma_1,\mu)$ and $(\Omega_2,\Sigma_2,\nu)$ be finite atomless measure spaces. Then, there is some $(\mu\otimes\nu)$-measurable set $E$ with $(\mu\otimes\nu)(E)>0$, such that $$(\mu\otimes\nu)((A\times B)\cap E)<(\mu\otimes \nu)(A\times B), \quad \forall A\times B\in \Sigma_1\times \Sigma_2 \text{ with } \mu(A)\nu(B)>0.$$
\end{lema}

\begin{proof}
    Assume that $\mu(\Omega_1)=1=\nu(\Omega_2)$ and consider $g_\mu: \Omega_1\rightarrow [0,1]$ and $g_\nu: \Omega_2\rightarrow [0,1]$ the associated measure-preserving maps from Lemma \ref{lema:Non-ConstantFunctionFiniteNonAtomicCase}. Thanks to \cite[251L]{Fremlin}, we have that the map $T: \Omega_1\times \Omega_2\rightarrow [0,1]^2$ given by $T=(g_\mu, g_\nu)$ is also measure-preserving. That is, for every $(\mu\otimes\nu)$-measurable set $F$, we have
    $$(\mu\otimes \nu)\left(T^{-1}(F)\right)=\lambda_2(F), $$
    where $\lambda_2$ is the Lebesgue measure in $[0,1]^2$. Invoking \cite[Lemma 2]{Choi}, the $\lambda_2$-measurable set
    $S=\{(x,y)\in [0,1]^2: |x-y|\in \Delta\}$
    (where $\Delta$ denotes any Cantor-type set of positive measure) satisfies that $\lambda_2(S)=\lambda(\Delta)>0$, and $$\lambda_2((C\times D)\cap S)<\lambda_2(C\times D), \quad \forall \ C\times D\subseteq [0,1]^2 \text{ with } \lambda_2(C\times D)>0.$$ Define the $(\mu\otimes\nu)$-measurable set $E\coloneqq T^{-1}(S)$. It is clear that $(\mu\otimes\nu)(E)=\lambda_2(S)>0$. 
    Let $A\in \Sigma_1$ and $B\in \Sigma_2$ with $\mu(A)\nu(B)>0$ and we are going to prove that $(\mu\otimes\nu)((A\times B)\setminus E)>0$. Consider the following measures $\widetilde{\mu}$ and $\widetilde{\nu}$ in $[0,1]$ given by
    \begin{align*}
        \widetilde{\mu}(R)&\coloneqq \mu\left(A\cap g_\mu^{-1}(R)\right),
        \\ \widetilde{\nu}(R)&\coloneqq \nu\left(B\cap g_\nu^{-1}(R)\right),
    \end{align*}
    for all $\lambda$-measurable $R\subseteq[0,1]$. The measure-preserving condition on $g_\mu$ and $g_\nu$ ensures that $\widetilde{\mu}$ and $\widetilde{\nu}$ are absolutely continuous with respect to $\lambda$. Denote by $f$ and $h$ the Radon-Nikodým derivatives of $\widetilde{\mu}$ and $\widetilde{\nu}$ respectively. They are clearly non-zero, because so are $\widetilde{\mu}$ and $\widetilde{\nu}$ (notice that $\widetilde{\mu}([0,1])=\mu(A)>0$ and $\widetilde{\nu}([0,1])=\nu(B)>0$). Hence, we can find $\varepsilon>0$ and some $\lambda$-measurable sets $C,D \subseteq[0,1]$ with $f\geq \varepsilon\chi_{C}$ and $h\geq \varepsilon\chi_{D}$. We claim that 
    \begin{equation*}
        (\mu\otimes\nu)\left((A\times B)\cap T^{-1}(F) \right)=\int_F f(\omega)h(\gamma) \text{ d}\lambda_2,\quad \text{ for all $\lambda_2$-measurable } F \subseteq [0,1]^2.
    \end{equation*}
    First, if $F$ is of the form $R_1\times R_2$, then
    \begin{align*}
        (\mu\otimes\nu)\left((A\times B)\cap T^{-1}(F) \right)&=(\mu\otimes\nu)\left( \left(A\cap g_\mu^{-1}(R_1)\right) \times \left(B\cap g_\nu^{-1}(R_2)\right) \right)\\&=\mu\left(A\cap g_\mu^{-1}(R_1)\right)\nu\left(B\cap g_\nu^{-1}(R_2)\right)=\widetilde{\mu}(R_1)\widetilde{\nu}(R_2)\\&=\left(\int_{R_1} f(\omega) \text{ d}\lambda \right)\left(\int_{R_2} h(\gamma) \text{ d}\lambda \right) =\int_F f(\omega)h(\gamma) \text{ d}\lambda_2.
    \end{align*}
    Furthermore, it is easy to see that the sets verifying the claim form a Dynkin system (or $\lambda$-system), and clearly the set of measurable rectangles form a $\pi$-system. Therefore, we derive the conclusion from the Dynkin's $\pi$-$\lambda$ theorem (see e.g. \cite[136B]{Fremlin}). Finally, we have
    \begin{align*}
        (\mu\otimes \nu)\left((A\times B)\setminus E \right) &= (\mu\otimes \nu) \left( (A\times B)\cap T^{-1}([0,1]^2\setminus S)\right) \\&=\int_{[0,1]^2\setminus S} f(\omega)h(\gamma) \text{ d}\lambda_2 \geq \varepsilon^2\lambda_2((C\times D)\setminus S)>0,
    \end{align*}
    finishing the proof.
\end{proof}

The following example presents two similarly constructed elements that might a priori be expected to behave analogously. However, we show that one attains the norm, whereas the other fails to do so.

\begin{ejem}
    Consider the space $L_1([0,1]^2)$, and, as in the proof of Lemma \ref{lema:setwithoutrectangles}, let $E=\{(x,y)\in [0,1]^2 : |x-y|\in \Delta\}$, where $\Delta$ is any Cantor-type set of positive measure. It is clear that $E$ is closed, so its complementary $E^c$ is open in $[0,1]^2$, and then it is a countable union of measurable rectangles. Therefore, we obtain as a consequence of Theorem \ref{thm:L1L1characterization}, the following contrasting examples:
    \begin{enumerate}
        \item $f=2\chi_{E}+\chi_{E^c}\notin \NA_\pi(L_1(\mu)\pten L_1(\nu))$, because $\{(\omega,\gamma) \in [0,1]^2 : f(\omega,\gamma)>1\} = E$ is not a countable union of measurable rectangles (Lemma \ref{lema:setwithoutrectangles} or \cite[Lemma 2]{Choi}). 
        \item $f=\chi_{E}+2\chi_{E^c}\in \NA_\pi(L_1(\mu)\pten L_1(\nu))$, because $\{(\omega,\gamma) : f(\omega,\gamma)>c\}\in \{ \emptyset, E^c, [0,1]^2\}$, for all $c\geq0$, so they are all of them countable unions of measurable rectangles. Indeed, writing $E^c=\bigcup_{i\in I} A_i\times B_i$, where $I\subseteq\N$ and $(A_i\times B_i)_{i\in I}\subseteq [0,1]^2$ are pairwise disjoint rectangles, we have that the representation $f=\chi_{[0,1]}\otimes \chi_{[0,1]} + \sum_{i\in I} \chi_{A_i}\otimes \chi_{B_i}$ is optimal.
    \end{enumerate}
\end{ejem}

Now, we prove the promised statement of this section. This theorem is stated in full generality for arbitrary measures. When the measures are not $\sigma$-finite, we consider the product measure defined in \cite[251F]{Fremlin}, since it satisfies the desired identification $L_1(\mu)\pten L_1(\nu)=L_1(\mu\otimes \nu)$ \cite[253F and 253L]{Fremlin}.

\begin{thm}\label{thm:L1L1} Let $(\Omega_1,\Sigma_1,\mu)$ and $(\Omega_2,\Sigma_2,\nu)$ be measure spaces. Then the following are equivalent:
\begin{itemize}
    \item[i)] $\mu$ or $\nu$ is a purely atomic measure.
    \item[ii)] $\NA_\pi(L_1(\mu)\pten L_1(\nu))=L_1(\mu)\pten L_1(\nu)$.     
\end{itemize}
\end{thm}

\begin{proof}
    i) $\Rightarrow$ ii) is clear.

    For ii) $\Rightarrow$ i), suppose that $\mu$ and $\nu$ are both non purely atomic. Hence, there is some measurable rectangle $A\times B\in \Sigma_1\times \Sigma_2$ with $0<\mu(A), \ \nu(B) < \infty$ and satisfying that both $A$ and $B$ do not contain atoms for the measures $\mu$ and $\nu$ respectively. Denote by $g_\mu$ and $g_\nu$ the corresponding measurable functions from Lemma~\ref{lema:Non-ConstantFunctionGeneralCase} associated to the previous sets. It is clear that $\widetilde{\mu}\coloneqq \mu|_A$ and $\widetilde{\nu}\coloneqq \nu|_B$ are finite atomless measures, so we can find, using Lemma \ref{lema:setwithoutrectangles}, some $(\widetilde{\mu}\otimes \widetilde{\nu})$-measurable set $E$ with $(\widetilde{\mu}\otimes \widetilde{\nu})(E)>0$ such that
    $$(\widetilde{\mu}\otimes\widetilde{\nu})((C\times D)\cap E)<(\widetilde{\mu}\otimes \widetilde{\nu})(C\times D), \quad \forall \ C\times D\in \Sigma_1\times \Sigma_2 \text{ with } \widetilde{\mu}(C)\widetilde{\nu}(D)>0.$$
    Observe that replacing $E$ by $E\cap (A\times B)$, we may assume that $E\subseteq A\times B$.
    Thus, it is clear that $\chi_E\in L_1(\widetilde{\mu}\otimes \widetilde{\nu})$ does not attain its norm, thanks to Corollary~\ref{coro:equivalencecharacteristicfunctions}. Furthermore, $E$ is also $(\mu\otimes\nu)$-measurable and since $E\subseteq A\times B$, we have $\chi_E\in L_1(\mu\otimes \nu)$. Finally, since $L_1(\mu)=L_1(\mu|_A)\oplus_1 L_1(\mu|_{\Omega_1\setminus A})$ and $L_1(\nu)=L_1(\nu|_B)\oplus_1 L_1(\nu|_{\Omega_2\setminus B})$ we conclude that $\chi_E\notin \NA_\pi (L_1(\mu)\pten L_1(\nu))$, in virtue of \cite[Lemma~3.1]{GGR}.
\end{proof}

\begin{rema}
    There is an alternative proof of Theorem \ref{thm:L1L1}. Indeed, a slight modification of \cite[Lemma 2.1]{saleh1} allows us to find, for any cardinals $\alpha, \beta$, a measurable set $S\in [0,1]^\alpha \times [0,1]^\beta$ such that $\widetilde{\lambda}(S\cap (A\times B))<\widetilde{\lambda} (A\times B)$ for all measurable sets $A\subseteq [0,1]^\alpha$ and $B\subseteq[0,1]^\beta$, where $\widetilde{\lambda}$ is the product Lebesgue measure on $[0,1]^\alpha\times[0,1]^\beta$. Hence, $\chi_S\notin \NA_\pi (L_1([0,1]^\alpha)\pten L_1([0,1]^\beta))$ in virtue of Corollary \ref{coro:equivalencecharacteristicfunctions}. This, combined with \cite[p. 501]{deflo}, \cite[Theorem 14.9]{Lacey} and \cite[Lemma 3.1]{GGR} finishes the proof.

    The advantage of the proof given in Theorem \ref{thm:L1L1} is that it provides a more descriptive explanation of the phenomenon of norm-attainment.
\end{rema}

In the last result of this section, we leverage the identification of norm-attaining elements to establish a topological property of the set $\NA_\pi(L_1(\mu)\pten L_1(\nu))$, for finite atomless measures.

\begin{prop}\label{prop:L1L1meager}
    Let $(\Omega_1, \Sigma_1,\mu)$ and $(\Omega_2,\Sigma_2,\nu)$ be finite atomless measure spaces. Then, the set $\NA_\pi(L_1(\mu)\pten L_1(\nu))$ is meager. 
\end{prop}

\begin{proof}
   Assume that $\mu(\Omega_1)=1=\nu(\Omega_2)$. Observe that if $f\in \NA_\pi(L_1(\mu)\pten L_1(\nu))$, then Lemma \ref{lema:NAL1L1} easily yields $|f|\in \NA_\pi(L_1(\mu)\pten L_1(\nu))$ (the representation $|f|=\sum_{i=1}^\infty |f_i|\otimes|g_i|$ is optimal, whenever $\sum_{i=1}^\infty f_i\otimes g_i$ is an optimal representation of $f$). Hence, thanks to Theorem \ref{thm:L1L1characterization}, we have
   $$\NA_\pi(L_1(\mu)\pten L_1(\nu)) \subseteq \bigcup_{q,r\in \mathbb{Q}_{>0}}A^r_{q},$$
   where \begin{align*}
       A_q^r&=\{f\in L_1(\mu\otimes\nu): \exists A\times B\in \Sigma_1\times\Sigma_2 \text{ with } \mu(A),\nu(B)\geq q \text{ and } |f|\geq r\chi_{A\times B}\}
   \end{align*}
   for each $q,r\in \mathbb{Q}_{>0}$. Hence, to see that $\NA_\pi(L_1(\mu)\pten L_1(\nu))$ is meager, it suffices to prove that the sets $A_q^r$ are all nowhere dense. Fix $q,r\in \mathbb{Q}_{>0}$. Let $\varepsilon>0,$ $f\in L_1(\mu\otimes\nu)$,  and we have to find $\widetilde{f}\in L_1(\mu\otimes\nu)$ and $\delta>0$ such that $\norm{f-\widetilde{f}}<\varepsilon$ and $B(\widetilde{f},\delta)\cap A_q^r=\emptyset$ (where $B(\widetilde{f},\delta)$ stands for the open ball centered at $\widetilde{f}$ with radius $\delta$). Since $f\in L_1(\mu\otimes\nu)$, pick $\varepsilon'>0$ such that if $(\mu\otimes\nu)(F)<\varepsilon'$, then $\int_F |f(\omega,\gamma)| \text{ d}(\mu\otimes\nu)<\varepsilon$. Taking $\Delta$ to be a (sufficiently) fat Cantor set in the proof of Lemma \ref{lema:setwithoutrectangles}, we may find a $(\mu\otimes\nu)$-measurable set $E$ such that $$(\mu\otimes\nu)(E^c)=1-(\mu\otimes\nu)(E)=1-\lambda(\Delta)<\varepsilon'$$ and $$(\mu\otimes\nu)((A\times B)\cap E)<(\mu\otimes \nu)(A\times B), \quad \forall A\times B\in \Sigma_1\times \Sigma_2 \text{ with } \mu(A)\nu(B)>0.$$ Define $\delta\coloneqq \inf\{ (\mu\otimes\nu)((A\times B)\cap E^c) : \mu(A),\nu(B)\geq q\}$, and we claim that $\delta>0$. Otherwise, we can find $(A_n\times B_n)_{n=1}^\infty \subseteq \Sigma_1\times \Sigma_2$ such that $\mu(A_n),\nu(B_n)\geq q$, for all $n\in \N$ and $\lim_{n\to\infty}(\mu\otimes\nu)((A_n\times B_n)\cap E^c)=0$. Since $(\chi_{A_n})_{n=1}^\infty\subseteq B_{L_\infty(\mu)}$ and $(\chi_{B_n})_{n=1}^\infty \subseteq B_{L_\infty(\nu)}$, there are subnets $(\chi_{A_\alpha})_\alpha$ and $(\chi_{B_\beta})_\beta$ that converge in the weak* topology to some $g\in L_\infty(\mu)$ and $h\in L_\infty(\nu)$, respectively. First, observe that
   \begin{align*}
       \int_{\Omega_1} g(\omega) \text{ d}\mu&=\langle g,1\rangle = \lim_\alpha \langle \chi_{A_\alpha},1\rangle =\lim_\alpha\int_{\Omega_1} \langle\chi_{A_\alpha},1\rangle \text{ d}\mu=\lim_\alpha\mu(A_\alpha)\geq q>0,
       \\ \int_{\Omega_1} h(\gamma) \text{ d}\nu&=\langle h,1\rangle = \lim_\beta \langle \chi_{B_\beta},1\rangle =\lim_\beta\int_{\Omega_2} \langle\chi_{B_\beta},1\rangle \text{ d}\nu=\lim_\beta\nu(B_\beta)\geq q>0.
   \end{align*}
   Then, we can find $s,t>0$ such that $C=\{\omega\in \Omega_1: g(\omega)\geq s\}$ and $D=\{\gamma\in \Omega_2: h(\gamma)\geq t\}$ have both positive measure. It follows that
   \begin{align*}
       \int_{E^c} g(\omega)h(\gamma) \text{ d}(\mu\otimes\nu) &\geq \int_{E^c\cap (C\times D)} g(\omega)h(\gamma) \text{ d}(\mu\otimes\nu)\\&\geq st(\mu\otimes\nu)(E^c\cap (C\times D))>0.
   \end{align*}
   Furthermore, it is not hard to see that $(\chi_{A_\alpha}\otimes\chi_{B_\beta})_{(\alpha,\beta)}$ converges weak* to $g\otimes h\in L_\infty(\mu\otimes\nu)$. Hence, 
   \begin{align*}
       \int_{E^c} g(\omega)h(\gamma) \text{ d}(\mu\otimes\nu)&=\lim_{\alpha,\beta}\int_{E^c} \chi_{A_\alpha}(\omega)\chi_{B_\beta}(\gamma) \text{ d}(\mu\otimes\nu)\\&=\lim_{\alpha,\beta}(\mu\otimes\nu)((A_\alpha\times B_\beta)\cap E^c)=0,
   \end{align*}
   obtaining a contradiction. Thus, $\delta>0$. 

   Next, consider $\widetilde{f}\coloneqq f|_E$. On the one hand,
   $$\norm{f-\widetilde{f}}=\int_{E^c} |f(\omega,\gamma)| \text{ d}(\mu\otimes\nu) <\varepsilon.$$
   On the other hand, given $\widehat{f}\in A_q^r$, there are $A\times B\in \Sigma_1\times \Sigma_2$ with $\mu(A),\nu(B)\geq q$ and $|\widehat{f}|_{A\times B}|\geq r$, so
   $$\norm{\widehat{f}-\widetilde{f}}\geq \int_{E^c\cap (A\times B)} |\widehat{f}(\omega,\gamma)| \text{ d}(\mu\otimes\nu) \geq r\mu(E^c\cap (A\times B))\geq r\delta.$$
   In conclusion, $B(\widetilde{f},r\delta)\cap A_q^r=\emptyset$, finishing the proof.
\end{proof}

\section{Spaces \texorpdfstring{$Y$}{Y} for which every tensor in \texorpdfstring{$L_1(\mu)\pten Y$}{L1(mu,Y)} attains its norm}\label{sect:CCG}

In this section, we restrict our attention to the case $\mathbb K=\mathbb R$. As we will show in Remark \ref{rema:CCGcomplex}, our techniques do not apply to the complex case.

Recall that there are Banach spaces $Y$ such that $\NA_\pi(L_1(\mu)\pten Y)=L_1(\mu)\pten Y$, the simplest one is $Y=\ell_1(I)$, since $\NA_\pi(X\pten \ell_1(I))=X\pten\ell_1(I)$ for any Banach space $X$. More generally, since $L_1(\mu)$ is complemented in its bidual, Corollary 3.5 in \cite{GGR} shows that $\NA_\pi(L_1(\mu)\pten Y^*)=L_1(\mu)\pten Y^*$ whenever $Y$ is a subspace of a space $Z$ such that $Z^*=\ell_1(I)$ isometrically. This includes the case where $Y$ is a finite-dimensional space with finitely many extreme points in the ball as well as some Lipschitz-free spaces (e.g. if $M$ is a countable compact metric space). A key point in that result is the existence of an adjoint quotient operator $Q\colon \ell_1(I)\to Y^*$, so for any $y^*\in Y^*$ there is a preimage with the same norm (that is, $y^*$ is a ``convex series'' of the points $Q(e_i)$). That is the motivation for the following definition.

\begin{defi}\label{def:CCG} Let $Y$ be a real Banach space. We say that $Y$ is \emph{countably convexly generated (CCG)} if there is a sequence $(e_n)_{n\in \N}\subseteq S_Y$ such that for each $y\in S_Y$ we can find $(\lambda_n)_{n\in \N}\subseteq [0,+\infty)$ with $\sum_{n=1}^\infty \lambda_n=1$ and $y=\sum_{n=1}^\infty \lambda_n e_n$.
\end{defi}

Recall that an \textit{extreme point} of a convex set $C$ is a point that cannot be expressed as a nontrivial convex combination of different elements of $C$. The set of extreme points of $C$ is denoted by $\ext (C)$.

\begin{rema} A Banach space $Y$ has the \emph{Convex Series Representation Property} (CSRP) \cite{ALS} if for every $y\in S_Y$ there are $(\lambda_n)_n\subseteq [0,1]$ with $\sum_{n=1}^\infty \lambda_n=1$ and $(e_n)_n\subseteq\ext B_Y$ such that $y=\sum_{n=1}^\infty \lambda_n e_n$. Clearly, if $Y$ has the CSRP and $\ext(B_Y)$ is countable, then $Y$ is CCG. Conversely, if $Y$ is CCG then $\ext(B_Y)$ is countable (but it might be empty, as the examples below show, thus failing the CSRP). 

Moreover, in \cite[Theorem 1]{ALS} it is shown that $Y$ has the CSRP if and only if it has \emph{$\lambda$-property}, that is, for each $x\in B_Y$ there are $e\in \ext B_Y$, $y\in B_Y$ and $\lambda \in (0,1]$ such that $x=\lambda e+ (1-\lambda)y$. The same proof shows that $Y$ is CCG if and only if there is a countable set $D\subseteq S_Y$ such that 
\[ \forall x\in B_Y\, \exists y\in B_Y, e\in D, \lambda\in (0,1] \text{ such that } x=\lambda e+(1-\lambda)y\]
\end{rema}

We provide now some more examples. Recall that for a topological space $S$, we say that $S$ is  \textit{scattered} if every non-empty subset $A\subseteq S$ contains an isolated point in $A$. We also say that $S$ is \textit{totally disconnected} if its connected subsets are just the singletons. Furthermore, we refer the reader to \cite{Weaver} for the definition and properties of the Lipschitz-free space (also called Arens-Eells space) over a metric space.

\begin{ejem}\label{ex:CCG} The following spaces are CCG.
\begin{enumerate}[a)]
    \item The Lipschitz-free space $\mathcal F(M)$ over a complete countable  metric space  $M$.  
    \item $C(K)$, where $K$ is a compact metrizable totally disconnected space. 
    \item The spaces $c_0$, $c$ and $\ell_1$.  
      \item $(\bigoplus_{n=1}^\infty Y_n)_0$ and $(\bigoplus_{n=1}^\infty Y_n)_1$, where each $Y_n$ is CCG.    
\end{enumerate}
\end{ejem}

\begin{proof}
a) Since $M$ is complete and countable, it is scattered, and thus every element in $\mathcal F(M)$ is a convex series of molecules \cite[Corollary 4.3]{APS26}. Also, the set of molecules is countable since so is $M$.

b) Recall that the  extreme points of $B_{C(K)}$ are precisely the continuous functions $f\colon K\to\{-1,1\}$. Thus, there is a bijection between the extreme points (up to the sign) and the clopen subsets of $K$. The hypotheses on  $K$ imply that it is zero-dimensional, that is, clopen sets are a basis. But $K$ is also second countable since it is metrizable. Therefore, there are countably many clopen subsets of $K$. That is, $\ext B_{C(K)}$ is a countable set. 

Finally, if $K$ is compact and totally disconnected, then $C(K)$ has the (uniform) $\lambda$-property and so the CSRP (see \cite{AB, Oates}). We include here an argument for the reader's convenience. Let $f\in B_{C(K)}$. Consider the disjoint closed sets $A=\{t\in K: f(t)\geq1/2\}$ and $B=\{t: f(t)\leq-1/2\}$. Since $K$ is zero-dimensional, there is a clopen set $U$ with $A\subseteq U\subseteq K\setminus B$. Let 
\[e=\chi_U-\chi_{K\setminus U}\in \ext B_{C(K)}, \qquad g=\frac{4}{3}\left(f-\frac{1}{4}e\right)\] 
Note that $\norm{g}_\infty\leq 1$. Indeed, if $t\in U$ then $-1/2<f(t)\leq 1$ and so 
\[ \left|f(t)-\frac{1}{4}e(t)\right|=\left|f(t)-\frac{1}{4}\right|\leq \frac{3}{4}\]
Analogously, if $t\notin U$ then $-1\leq f(t)<\frac{1}{2}$ and then
\[ \left|f(t)-\frac{1}{4}e(t)\right|=\left|f(t)+\frac{1}{4}\right|\leq \frac{3}{4}\]
Finally, $f=\frac{1}{4}e+\frac{3}{4}g$. This finishes the proof. 

c) For $\ell_1$, the statement is trivial. Also, the space $c=C(\mathbb N\cup\{\infty\})$ is CCG by b). Now we prove the case of $c_0$. Given a finite subset $F\subseteq \mathbb N$ and $\varepsilon =(\varepsilon_n)_n\subseteq \{-1,1\}^F$, we denote $e_{F, \varepsilon}:= \sum_{n\in F} \varepsilon_n e_n$. Take $y\in c_0$ with $\norm{y}=1$ and fix a permutation $\pi\colon \mathbb N\to\mathbb N$ such that $\lambda_n:=|y_{\pi(n)}|- |y_{\pi(n+1)}|\geq 0$ for each $n$. Note that $\sum_{n=1}^\infty \lambda_n= |y_{\pi(1)}|=\norm{y}=1$. Take $F_n=\{\pi(k):k\leq n\}$,  $\varepsilon_{\pi(k)} = \sign y_{\pi(k)}$ and $\varepsilon^{n} = (\varepsilon_k)_{k\in F_n}$. Then it follows that
\[ y = \sum_{n=1}^\infty \lambda_n e_{F_n, \varepsilon^n}\]
Also, it is clear that the set $\left\{e_{F, \varepsilon}: F \text{ is finite and } \varepsilon\in \{-1,1\}^F\right\}$ is countable.

d) It can be proved with a similar argument as in c).
\end{proof}

The following is the main result of this section. It enlarges the known cases where every tensor attains its norm. 

\begin{thm}\label{thm:CSRP}
    Let $(\Omega, \Sigma, \mu)$ be a measure space and let $Y$ be a real  Banach space. If $Y$ is CCG, then $$\NA_\pi (L_1(\mu)\pten Y)= L_1(\mu)\pten Y.$$
\end{thm}

We will obtain Theorem \ref{thm:CSRP} as a consequence of a measurable selection theorem by Cascales and Raja \cite{CR} (see also Remark \ref{rema:directproofCCG} for an alternative and more direct proof of Theorem \ref{thm:CSRP}). To this end, we need first some preparatory lemmas. The first one is the following relation between quotients and proximinality, that is well known to specialists, and it is a direct consequence of \cite[Theorem 2.3]{singer}. Recall that a subspace $Z$ of a Banach space $X$ is said to be \textit{proximinal} if for every $x\in X$ there is some $z\in Z$ which minimizes the distance from $x$ to $Z$. Moreover, a bounded operator $Q\colon X\to Y$ is  a \textit{quotient operator} if $Q$ is onto and $\norm{y}=\inf\{\norm{x}: Qx=y\}$ for each $y\in Y$.

\begin{lema}\label{lemma:proximinal}
 Let $X, Y$ be Banach spaces and $Q\colon X\to Y$ be a quotient operator. The following are equivalent:
\begin{itemize}
    \item[i)] $\ker Q$ is a proximinal subspace of $X$. 
    \item[ii)] For every $y\in Y$ there is $x\in X$ such that $Qx=y$ and $\norm{x}=\norm{y}$.
\end{itemize}
\end{lema}

\begin{lema}\label{lemma:kerIT} Let $(\Omega, \Sigma, \mu)$ be a measure space, and $Y, Z$ be Banach spaces. Given $I\colon L_1(\mu)\to L_1(\mu)$  the identity map, and $T\colon Y\to Z$ be a bounded operator, let $I\otimes T\colon L_1(\mu,Y)\to L_1(\mu, Z)$. Then 
\[\ker(I\otimes T)=L_1(\mu, \ker T)\]
\end{lema}

\begin{proof}
Note that $R\colon L_1(\mu, Y)\to L_1(\mu,Z)$ given by $Ru(\omega)=T(u(\omega))$
is a well-defined bounded linear operator. Also, $R$ coincides with $I\otimes T$  on the basic tensors $f\otimes y$ with $f\in L_1(\mu)$ and $y\in Y$. Thus, $I\otimes T=R$. 

Now, 
\begin{align*}
\ker R &= \{u\in L_1(\mu, Y): T(u(\omega))=0, \ \mu-\text{a.e.}\}\\
&=\{u\in L_1(\mu, Y): u(\omega)\in \ker T, \, \mu-\text{a.e}\}= L_1(\mu, \ker T)
\end{align*}
\end{proof}

Cascales and Raja proved in \cite[Theorem 3.4 and Remark 3.10]{CR} that if $(\Omega, \Sigma,\mu)$ is a complete probability space and $Z$ is a weakly countably determined (WCD) subspace of $X$, then $Z$ is proximinal in $X$ if and only if $L_1(\mu, Z)$ is proximinal in $L_1(\mu, X)$. In particular, separable or reflexive Banach spaces are examples of WCD spaces (see \cite[Remark 3.10]{CR} for the definition and further comments).

\begin{thm}\label{thm:quotientwKa} Let $(\Omega, \Sigma, \mu)$ be a measure space and let $X, Y$ be real Banach spaces. Assume that there is a quotient operator $Q\colon X\to Y$ such that $\ker Q$ is WCD and proximinal in $X$. If $\NA_\pi(L_1(\mu)\pten X)=L_1(\mu)\pten X$, then $\NA_\pi(L_1(\mu)\pten Y)=L_1(\mu)\pten Y$.
\end{thm}

\begin{proof}
First assume that $(\Omega,\Sigma,\mu)$ is complete and finite. Let $Z=\ker Q$, which is a proximinal subspace of $X$ and it is WCD.  It follows from \cite[Theorem 3.4]{CR}  that $L_1(\mu, \ker Q)$ is proximinal in $L_1(\mu, X)$. But note that, by Lemma \ref{lemma:kerIT}, $L_1(\mu, \ker Q)=\ker(I\otimes Q)$. 

Now we will show that $I\otimes Q$ is a quotient operator. Given $u\in L_1(\mu, Y)$ and $\varepsilon>0$, write $u=\sum_{i=1}^\infty f_i\otimes y_i$ with $\sum_{i=1}^\infty\norm{f_i}\norm{y_i}\leq \norm{u}+\varepsilon$.
Using Lemma \ref{lemma:proximinal}, we can find vectors $x_i\in X$ with $Qx_i=y_i$ and $\norm{x_i}=\norm{y_i}$. Then clearly $v=\sum_{i=1}^\infty f_i\otimes x_i$ satisfies $(I\otimes Q)v= u$ and $\norm{v}\leq \norm{u}+\varepsilon$. 

Therefore, by Lemma \ref{lemma:proximinal} we have that for each $u\in L_1(\mu, Y)$ there is $v\in L_1(\mu, X)$ such that $(I\otimes Q)v=u$ and $\norm{u}=\norm{v}$. It follows readily that if $v=\sum_{i=1}^\infty f_i\otimes x_i$ is an optimal representation for $v$, then $u=\sum_{i=1}^\infty f_i\otimes Qx_i$ is an optimal representation for $u$.     
\\

For the case of a general measure $\mu$, let $u\in L_1(\mu,Y)$. Passing to the completion of
$\mu$ does not change $L_1(\mu,Y)$ isometrically, so we may assume that $\mu$ is complete. Let
\[
E_n=\{\omega:\|u(\omega)\|>1/n\}.
\]
Then $\mu(E_n)<\infty$ and $u=0$, $\mu-$a.e. outside $\cup_n E_n$. Taking a
measurable partition $A_n$ of $\bigcup_n E_n$ with each $\mu(A_n)<\infty$,
we have
\[
u=\sum_{n=1}^\infty u\chi_{A_n},\qquad
\|u\|=\sum_{n=1}^\infty\|u\chi_{A_n}\|.
\]
By the finite case, each $u\chi_{A_n}$ has an optimal representation on
$A_n$. Extending the scalar functions by zero and concatenating the
representations gives an optimal representation of $u$ in $L_1(\mu)\pten Y$.
\end{proof}

\begin{proof}[Proof of Theorem \ref{thm:CSRP}] It follows from Theorem \ref{thm:quotientwKa} taking $X=\ell_1(\mathbb N)$, since every separable space is WCD. 
\end{proof}

\begin{rema}\label{rema:directproofCCG}
    Indeed, one can provide an alternative proof of Theorem \ref{thm:CSRP} that does not rely on Theorem \ref{thm:quotientwKa}. The idea of the proof is the following.

    First, with the same reasoning which was done in the proof of Theorem \ref{thm:quotientwKa} we can restrict our attention to complete finite measures. Then, since $Y$ is CCG, we can define the set-valued function $F\colon  B_Y\rightrightarrows B_{\ell_1(\N)}$  given by $$F(y)\coloneqq \left\{ (\lambda_n)_n\in B_{\ell_1(\N)} : y=\sum_{n=1}^\infty \lambda_n e_n\right\}, \quad \forall y\in B_Y.$$ It is easy to see that we can apply \cite[Theorem 6.9.2]{Bogachev} to obtain a (selection) map $f: B_Y \rightarrow B_{\ell_1(\N)}$ which is measurable with respect to the $\sigma$-algebra generated by all Souslin sets in $B_{\ell_1(\N)}$ and such that $f(y)\in F(y)$, for all $y\in B_Y$. Now, for $u\in L_1(\mu)\pten Y$, define $\widetilde{u}\in L_1(\mu)\pten Y$ given by $$\widetilde{u}(\omega)\coloneqq \begin{cases}
\frac{u(\omega)}{\norm{u(\omega)}}, & \text{if } \omega\in \supp u,\\
0, & \text{if } \omega\notin \supp u.
\end{cases}$$
    Using that $\mu$ is a complete finite measure, it can be seen that $f\circ \widetilde{u}$ is $\Sigma$-Borel-measurable. Finally, we define for each $n\in \N$, the functions $\lambda_n \colon \Omega \rightarrow \K$ such that
    $$\lambda_n(\omega) \coloneqq
\pi_n(f(\widetilde{u}(\omega)))\cdot \norm{u(\omega)}, \quad \omega \in \Omega,\ \mu- \text{a.e.},$$
where $\pi_n\colon \ell_1(\N)\rightarrow\K$ is the natural projection into the $n$-th coordinate. It follows that $u=\sum_{n=1}^\infty \lambda_n \otimes e_n$ and this is an optimal representation.

Observe that, although the proof is constructive because it provides the exact optimal representation, it depends on the measurable selection theorem \cite[Theorem~6.9.2]{Bogachev} which makes it difficult in practice to determine such a representation.
\end{rema}

We were not able to find any example showing that Theorem \ref{thm:quotientwKa} is more general than Theorem \ref{thm:CSRP}. 

Now we obtain some consequences of Theorem \ref{thm:CSRP} and Example \ref{ex:CCG}. 

\begin{coro}\label{cor:F(M)}
    Let $(\Omega,\Sigma,\mu)$ be a measure space and $M$ be a complete scattered metric space. Then,
    $$\NA_\pi(L_1(\mu)\pten \mathcal F(M))=L_1(\mu)\pten \mathcal F(M).$$
\end{coro}

\begin{proof}
    Let $u\in L_1(\mu)\pten \mathcal F(M)=L_1(\mu,\mathcal F(M))$. Since $u$ is strongly measurable, there is a separable subspace $Y\subseteq\mathcal F(M)$ such that $u(\omega)\in Y$,  $\mu-$a.e. A standard argument provides a countable subset $N\subseteq M$ such that $Y\subseteq \mathcal F(\overline{N})$ (take a dense sequence $(y_n)_n\subseteq Y$ and approximate it by finitely supported elements), moreover $\mathcal F(\overline{N})\subseteq \mathcal F(M)$ isometrically thanks to McShane's extension theorem.  Now, since $N$ is countable and $M$ is complete and scattered, we have that $\overline{N}$ is separable, complete and scattered. Furthermore, every  second countable and scattered space is countable \cite[Proposition 8.5.5]{Semadeni}, so $\overline{N}$ is also countable. 
    Therefore, $u$ has an optimal representation in $L_1(\mu)\pten \mathcal F(\overline{N})$ thanks to Theorem~\ref{thm:CSRP} and Example \ref{ex:CCG}. Finally, this representation is also optimal in $L_1(\mu)\pten \mathcal F(M)$, since $L_1(\mu)\pten \mathcal F(\overline{N})$ is (isometrically) a subspace of $L_1(\mu)\pten \mathcal F(M)$.
\end{proof}

Before exhibiting the next corollary, we refer the reader to \cite{StoneSpaces} and \cite{Fremlin} for basic facts on Boolean algebras and the Stone space. Recall that, for a compact Hausdorff space, being totally disconnected is equivalent to being zero-dimensional \cite[Theorem 4.2]{StoneSpaces}. 
 
\begin{coro}\label{cor:C(K)} Let $(\Omega, \Sigma, \mu)$ be a measure space and $K$ be a compact Hausdorff totally disconnected space. Then,
\[\NA_\pi(L_1(\mu)\pten C(K))= L_1(\mu)\pten C(K).\]
\end{coro}

\begin{proof}
Let $u\in L_1(\mu, C(K))$. Since $u$ is strongly measurable, there is a separable subspace $Y\subseteq C(K)$ such that $u(\omega)\in Y$, $\mu-$a.e. 

Let $(f_n)_n$ be a dense sequence in $Y$. Note that $C(K)=\overline{\spann}\{\chi_U: U\text{ is clopen in } K\}$ by Stone-Weierstrass theorem. 
Approximating each $f_n$, we get a countable family $\mathcal A_0$ of clopen sets such that $f_n\in \overline{\spann}\{\chi_U: U\in \mathcal A_0\}$ for each $n$. Now, let $\mathcal A$ be the Boolean algebra generated by $\mathcal A_0$, which is also countable, and $Z=\overline{\spann}\{\chi_U: U\in \mathcal A\}$. Then 
\[Y=\overline{\spann}\{f_n:n\in \mathbb N\}\subseteq  Z.\]  
and so $u\in L_1(\mu, Z)$.  Consider $L$ the Stone space of $\mathcal A$, that is, the set of ultrafilters on $\mathcal A$ endowed by the topology generated by the sets $\hat{A}=\{\mathcal U: A\in \mathcal U\}$, for $A\in \mathcal A$. By Stone's theorem, $L$ is compact Hausdorff and totally disconnected, and $Z=C(L)$ isometrically. Moreover, since $\mathcal A$ is countable, $L$ has a countable basis and so it is second countable. Thus, $L$ is metrizable, so $Z=C(L)$ is CCG by Example~\ref{ex:CCG}. Finally, $u$ has an optimal representation in $L_1(\mu)\pten Z$ by Theorem \ref{thm:CSRP}. Since $L_1(\mu)\pten Z$ is (isometrically) a subspace of $L_1(\mu)\pten C(K)$, that representation is also optimal in   $L_1(\mu)\pten C(K)$. 
\end{proof}

Since $L_\infty(\nu)$ is isometric to the space of continuous functions on the associated Stone space (see e.g. \cite[363A, 363I]{Fremlin}), we get:

\begin{coro}\label{cor:Linfty} For any measures $\mu$ and $\nu$, we have
\[\NA_\pi(L_1(\mu)\pten L_\infty(\nu))= L_1(\mu)\pten L_\infty(\nu).\]
\end{coro}

\begin{coro}\label{cor:c0} Let $(\Omega, \Sigma, \mu)$ be a measure space. Then,
\[\NA_\pi(L_1(\mu)\pten c_0(\Gamma))= L_1(\mu)\pten c_0(\Gamma).\]
\end{coro}    

\begin{proof} 
For countable $\Gamma$, it follows directly from Theorem \ref{thm:CSRP} since $c_0$ is CCG. Now, given $u\in L_1(\mu, c_0(\Gamma))$, let $(u_n)$ be a sequence of simple functions converging to $u$ in norm. Since each vector of $c_0(\Gamma)$ is countably supported, there is a countable set $\Gamma_0$ such that $u_n(\omega)\in c_0(\Gamma_0)$,  $\mu-$a.e. Thus $u(\omega)\in c_0(\Gamma_0)$,  $\mu-$a.e., and so $u\in L_1(\mu, c_0(\Gamma_0))$. Finally, the optimal representation of $u$ in $L_1(\mu)\pten c_0(\Gamma_0)$ is also optimal in $L_1(\mu)\pten c_0(\Gamma)$, since $L_1(\mu)\pten c_0(\Gamma_0)$ is (isometrically) a subspace of $L_1(\mu)\pten c_0(\Gamma)$. 
\end{proof}

This motivates the following question. Recall that a Banach space is \textit{polyhedral} if the unit ball of every finite-dimensional subspace is a polytope (e.g. $c_0(\Gamma)$).

\begin{question} Let $(\Omega, \Sigma,\mu)$ be a measure space and $Y$ be a polyhedral space. Is it true that 
\[\NA_\pi(L_1(\mu)\pten Y)=L_1(\mu)\pten Y ?\]
\end{question}

 Note that if $Y$ is polyhedral then $L_1(\mu)\otimes Y \subseteq \NA_\pi (L_1(\mu)\pten Y)$ (that is, finite rank tensors attain their norm). 
    
    Indeed, given $u\in L_1(\mu)\otimes Y$, there is some finite-dimensional subspace $E\subseteq Y$ such that $u\in L_1(\mu)\otimes E$. Since $B_E$ is a polytope and $L_1(\mu)$ is 1-complemented in the bidual, we have that $u\in \NA_\pi (L_1(\mu)\pten E)$, in virtue of \cite[Theorem 4.1]{DGLJRZ} and \cite[Corollary 3.2]{GGR}. Thus, the conclusion follows from the fact that $L_1(\mu)\pten E=L_1(\mu,E)$ is a closed subspace of $L_1(\mu)\pten Y=L_1(\mu,Y)$.

Finally, we highlight the difficulties in extending the techniques developed in this section to the complex-valued setting. If one wishes to consider the case $\K=\Complex$, then the natural analogue of Definition \ref{def:CCG} would be to allow the sequence of scalars $(\lambda_n)_n$ to be any complex sequence satisfying $\sum_{n=1}^\infty |\lambda_n| = 1$, that is, to replace convex series with absolutely convex series. However, the following remark illustrates how difficult it is to find complex Banach spaces satisfying this property.

\begin{rema}\label{rema:CCGcomplex}
    Let $Y=\ell_\infty^2(\Complex)$ (i.e. $Y$ is the 2-dimensional complex Banach space with the norm $\norm{\cdot}_\infty$). We claim that it is impossible to find a sequence $(e_n)_n\subseteq S_Y$ satisfying that for every $y\in S_Y$ there is a sequence $(\lambda_n)_n\subseteq \Complex$ with $y=\sum_{n=1}^\infty \lambda_n e_n$ and $\sum_{n=1}^\infty |\lambda_n|=1$.

    Indeed, suppose that such a sequence $(e_n)_n\subseteq S_Y$ exists. It is clear that we may assume the vectors $e_n$ to be pairwise non-colinear. In particular, for each $y=(y(1), y(2))\in S_Y$ with $|y(1)|=|y(2)|=1$, the associated sequence $(\lambda_n)_n\subseteq\Complex$ must satisfy
    \begin{align*}
        1=|y(1)|&=\left|\sum_{n=1}^\infty \lambda_n e_n(1)\right|\leq \sum_{n=1}^\infty |\lambda_n|=1,
        \\ 1=|y(2)|&=\left|\sum_{n=1}^\infty \lambda_n e_n(2)\right|\leq \sum_{n=1}^\infty |\lambda_n|=1.
    \end{align*}
    The above equalities and the pairwise non-colinearity of the vectors $e_n$ yields the existence of a unique $n\in \N$ such that $\lambda_n\neq 0$, and so $y=\lambda_n e_n$. This is impossible because there is a non-countable number of pairwise non-colinear elements $y$ with such form (e.g. the family $y_\theta=(1,\theta)$ with $|\theta|=1$).
\end{rema}

We conclude by providing a counterexample to \cite[Question 6.3]{DJRRZ}, where it is  asked whether the equality $X\pten Y=\NA_\pi(X\pten Y)$ holds whenever $X$ has property $\alpha$. This can be answered negatively as follows: Take $X$ and $Y$ to be any Banach spaces such that there is a tensor $u\in X\pten Y$ that does not attain the norm. By a result of Schachermayer \cite[Theorem 4.6]{Sch} there is a space $Z$ with property $\alpha$ such that $X$ is $1$-complemented in $Z$. Then $X\pten Y$ is isometrically a subspace of $Z\pten Y$, and $u\notin \NA_\pi(Z\pten Y)$ \cite[Lemma 3.1]{GGR}.

\section*{Acknowledgements} 

The authors would like to thank A. Procházka for several insightful discussions, which were essential in obtaining the results of Section 3. We are also grateful to Ramón J. Aliaga, Miguel Martín, José Rodríguez, Óscar Roldán, and Abraham Rueda Zoca for their useful comments. 

The research of Luis C. García-Lirola was supported by grants PID2021-122126NB-C31 and PID2022-137294NB-I00 funded by MCIN/AEI/\\ 
10.13039/501100011033 and by ``ERDF A way of making Europe''; and by grant E48-23R funded by Diputación General de Aragón (DGA), and by Fundaci\'on S\'eneca: ACyT Regi\'on de Murcia grant 21955/PI/22.

The research of Juan Guerrero-Viu was supported by FPU24/02284 predoctoral grant funded by MCIU; by grant PID2022-137294NB-I00 funded by MCIN/AEI/
10.13039/501100011033 and by ``ERDF A way of making Europe''; by grant E48-23R funded by Diputación General de Aragón (DGA).

The authors acknowledge the use of artificial intelligence–based tools, including ChatGPT (OpenAI), and Gemini (Google), for language editing, stylistic improvements and helpful discussions during the preparation of the manuscript. The authors remain fully responsible for the content of this work.

\end{document}